\documentclass{article}
\usepackage[preprint]{tmlr}

\usepackage{graphicx}
\usepackage{hyperref}
\hypersetup{
  hidelinks,
  hypertexnames=false,
  pdftitle={Dec-BFTRL: Improved Regret for Efficient Decentralized Online Upper-Linearizable Optimization with Separation Oracles},
  pdfauthor={Yiyang Lu}
}
\usepackage{url}

\DeclareRobustCommand{\algname}{\textnormal{\textsc{Dec-BFTRL}}}
\DeclareRobustCommand{\solvername}{\textnormal{\textsc{HybridNewton}}}

\usepackage{microtype}
\usepackage{subfigure}
\usepackage{booktabs}
\usepackage{wrapfig}
\usepackage{natbib}

\usepackage{braket}
\usepackage{algorithm}
\usepackage{algorithmic}
\usepackage{multirow}
\usepackage{makecell}
\usepackage{xcolor}
\usepackage{lmodern}
\usepackage{enumitem}
\let\AND\relax

\usepackage{amsmath}
\usepackage{amssymb}
\usepackage{mathtools}
\usepackage{amsthm}

\theoremstyle{plain}
\newtheorem{theorem}{Theorem}
\newtheorem{proposition}{Proposition}
\newtheorem{lemma}{Lemma}
\newtheorem{corollary}{Corollary}
\theoremstyle{definition}
\newtheorem{definition}{Definition}
\newtheorem{assumption}{Assumption}
\theoremstyle{remark}
\newtheorem{remark}{Remark}

\title{\raggedright Dec-BFTRL: Squre-Root Regret for Decentralized Online 
\\ Upper-Linearizable Optimization under Separation Access with Application to Continuous Submodular Maximization}
\author{\name Yiyang Lu \email yiyanglu@purdue.edu \\
     \addr Purdue University, West Lafayette, IN, USA\\
      \AND
      \name Mohammad Pedramfar \email mohammad.pedramfar@mila.quebec \\
      \addr Mila - Quebec AI Institute/McGill University, Montreal, QC, Canada\\
      \AND
      \name Vaneet Aggarwal \email vaneet@purdue.edu\\
      \addr  Purdue University, West Lafayette, IN, USA
 }

\begin{document}

\maketitle

\begin{abstract}
We study decentralized online optimization of upper-linearizable payoffs over an action set under efficient separation access, with applications to online continuous diminishing-return (DR) submodular maximization.
We propose Decentralized Barrier Follow-the-Regularized-Leader (Dec-BFTRL), and evaluate each agent's played action against the average of all local objectives. 
Each agent maps an internal iterate to a feasible action through an approximate gauge projection, communicates only a cumulative surrogate-gradient dual state, and invokes the local \solvername{} procedure to approximately minimize its post-communication BFTRL potential.
For every agent, we achieve expected network-aggregate regret of $\widetilde O(\sqrt{T})$. Over $T$ rounds, each agent uses $T$ neighbor-mixing steps and $\widetilde O(T)$ separation-oracle calls.
We give wrapper instantiations covering four up-concave or DR-submodular maximization problems.
\end{abstract}

\section{Introduction}

The demand for real-time decision-making in large-scale systems, such as autonomous robotic swarms and networked sensor arrays, has driven the rapid development of decentralized online optimization \citep{li2002detection, xiao2007distributed, mokhtari2018decentralized}. In these distributed frameworks, individual agents must collaboratively process continuous data streams and adapt to evolving environments without relying on a central coordinating server.
A fundamental theoretical and practical problem arises when the time-varying global objectives are highly non-convex. Fortunately, a wide array of critical network applications, including dynamic pricing, recommendation algorithms, inventory routing, mean-field variational inference, and power grid reconfiguration \citep{bian2019optimal, aldrighetti2021costs, ito2016large, hassani17_gradien_method_submod_maxim, mitra2021submodular, gu2023profit, mishra2017comprehensive}, exhibit structural regularities, most notably the economic property of diminishing returns (DR). In these scenarios, the marginal
benefit of an agent's action naturally decreases as the global system
approaches optimality, making cooperative online learning essential.
Although much of the literature models such objectives through continuous DR-submodularity or weak up-concavity, upper-linearizability provides a broader abstraction for designing unified algorithms across non-concave function classes.

\paragraph{Problem Setting and Metric.}
We consider \(N\) agents connected by a fixed, connected, undirected communication network. The agents share a compact convex action set \(\mathcal K\), which is accessed through a separation oracle and contains a known interior point. Before the game begins, an oblivious adversary fixes local payoff functions \(\{f_{t,j}\in\mathcal F:t\in[T],\,j\in[N]\}\), where \(j\) indexes the owner of a local payoff. At round \(t\), each agent \(a\) selects an action \(x_{a,t}\in\mathcal K\), receives feedback about its locally owned payoff \(f_{t,a}\) (which is not necessarily incurred at $x_{a,t}$), and exchanges information once with its neighbors before the next round.

Define the network-aggregate payoff $\mathsf F_t(x):=\frac1N\sum_{j=1}^N f_{t,j}(x).$
Although each agent receives only local feedback, its performance is measured using the aggregate payoff. For \(\alpha\in(0,1]\), the per-agent network-aggregate expected \(\alpha\)-regret is
\begin{align*}
\mathfrak R_\alpha^a(T):= \mathbb E\!\left[
\alpha\max_{x\in\mathcal K}\sum_{t=1}^T\mathsf F_t(x)
- \sum_{t=1}^T\mathsf F_t(x_{a,t})
\right].
\end{align*}
The expectation is over the randomness of the algorithm and its feedback. Section~\ref{sec:prelim} specifies the feedback interface, action-set geometry, and network assumptions.

\begin{table*}
    \centering
    \caption[Summary of D-OCSM Algorithms]{Summary of Decentralized Online Continuous DR-Submodular Maximization (D-OCSM) Algorithms. Let $c\in[0,1]$ denote the curvature in the monotone general-convex-set setting, and $\underline x\in\mathcal K$ denote the fixed feasible anchor used by Corollary~\ref{cor:nonmono_general}, with $\|\underline x\|_\infty<1$. `(Smooth)' means additional smoothness condition on the objective functions. Oracle-call counts follow the access model of each cited result, and SO and LOO calls should not be treated as equal-cost operations. The notation $\widetilde O(T)$ suppresses logarithmic dependence on $T$ and the set-conditioning parameters.}
    \label{tab:comparison}
    \resizebox{\textwidth}{!}{
    \begin{tabular}{llccccc}
        \toprule
        \textbf{Setting/Function Class} & \textbf{Reference/Algorithm} & \textbf{Approx. Ratio ($\alpha$)} & \textbf{$\log_T(\alpha$-Regret)} & \textbf{$\log_T$(Comm.)} & \textbf{Oracle Type} & \textbf{Total Oracle Calls} \\
        \midrule
        Monotone, $0 \in \mathcal{K}$
        & DMFW \citep{zhu2021projection} & $1-e^{-1}$ & $1/2$ & $5/2$ & LOO & $O(T^{5/2})$ \\
        & Mono-DMFW \citep{zhang2023communication} & $1-e^{-1}$ & $4/5$ & $1$ & LOO & $O(T)$ \\
        & DOBGA \citep{zhang2023communication} & $1-e^{-1}$ & $1/2$ & $1$ & Projection & - \\
        & DPOBGA \citep{liao2023improved} & $1-e^{-1}$ & $3/4$ & $1/2$ & LOO & $O(T)$ \\
        & DROCULO ($\theta=1$) \citep{lu2025decentralized} & $1-e^{-1}$ & $1/2$ & $1$ & LOO & $O(T^2)$ \\
        & DROCULO ($\theta=1/2$) \citep{lu2025decentralized} & $1-e^{-1}$ & $3/4$ & $1/2$ & LOO & $O(T)$ \\
        & Corollary~\ref{cor:mono_0_in_K} (this work) & $1-e^{-1}$ & $1/2$ & $1$ & SO & $\tilde O(T)$ \\
        \midrule
        Monotone, General $\mathcal K$
        & DROCULO ($\theta=1$) \citep{lu2025decentralized} & $\frac{1}{1+c}$ & $1/2$ & $1$ & LOO & $O(T^2)$ \\
        & DROCULO ($\theta=1/2$) \citep{lu2025decentralized} & $\frac{1}{1+c}$ & $3/4$ & $1/2$ & LOO & $O(T)$ \\
        & Corollary~\ref{cor:mono_general} (this work) & $\frac{1}{1+c}$ & $1/2$ & $1$ & SO & $\tilde O(T)$ \\
        \midrule
        Non-monotone, General $\mathcal{K}$ 
        & DROCULO ($\theta=1$) \citep{lu2025decentralized} & $\frac{1 - \|\underline{\mathbf{x}}\|_\infty}{4}$ & $1/2$ & $1$ & LOO & $O(T^2)$ \\
        & DROCULO ($\theta=1/2$) \cite{lu2025decentralized} & $\frac{1 - \|\underline{\mathbf{x}}\|_\infty}{4}$ & $3/4$ & $1/2$ & LOO & $O(T)$ \\
        & \citet{wan2026improved} & $\frac{1 - \|\underline{\mathbf{x}}\|_\infty}{4}$ & $3/4$ & $1$ & LOO & $O(T)$ \\
        & \citet{wan2026improved} (Smooth) & $\frac{1 - \|\underline{\mathbf{x}}\|_\infty}{4}$ & $2/3$ & $1$ & LOO & $O(T)$ \\
        & Corollary~\ref{cor:nonmono_general} (this work) & $\frac{1 - \|\underline{\mathbf{x}}\|_\infty}{4}$ & $1/2$ & $1$ & SO & $\tilde O(T)$ \\
        \midrule
        Non-monotone, Down-closed $\mathcal{K}$ 
        & \citet{wan2026improved} & $1/e$ & $3/4$ & $1$ & LOO & $O(T)$ \\
        & \citet{wan2026improved} (Smooth) & $1/e$ & $2/3$ & $1$ & LOO & $O(T)$ \\
        & Corollary~\ref{cor:nonmono_downclosed} (this work) & $1/e$ & $1/2$ & $1$ & SO & $\tilde O(T)$ \\
        \bottomrule
    \end{tabular}
    }
\end{table*}

\paragraph{Related Work.}
Recent development in \emph{decentralized online continuous submodular maximization} has produced several distinct regret-communication-computation tradeoffs \citep{zhu2021projection,zhang2023communication,liao2023improved, lu2025decentralized,wan2026improved}. Early algorithms focused on monotone continuous DR-submodular rewards, while \citet{lu2025decentralized} provided a decentralized framework for the broader upper-linearizable interface. Recent reductions sharpen several application-specific results for non-monotone continuous DR-submodular rewards \citep{wan2026improved}, specifically for smooth objectives. 
In the centralized setting, upper-linearizable algorithms and uniform wrappers provide the translation from online linear regret to a range of non-concave payoff models \citep{pedramfar2024linear,pedramfar2025uniform}, and are recently extended to down-closed DR-submodular maximization \citep{lu2026upper}.
For online convex optimization (OCO) over difficult action sets, another line of work replaces Euclidean projection with weaker action-set access \citep{hazan2012projection,garber2016linearly,levy2019projection, mhammedi2022efficient, garber2022new,mhammedi2025online} and proposed projection-free algorithms. Of particular relevance, \citet{mhammedi2025online} combine a separation-to-gauge reduction with a centralized Taylor Barrier-ONS analysis. 
This naturally motivates an oracle-specific question:
\begin{center}
    \emph{What performance can decentralized online upper-linearizable optimization achieve under separation access, and what is the resulting application in Decentralized Online Continuous Submodular Maximization?} 
\end{center}
We attain $\widetilde O(\sqrt T)$ regret, $O(T)$ communication, and $\widetilde O(T)$ separation-oracle calls per agent under the assumptions made precise in Section~\ref{sec:prelim}. Table~\ref{tab:comparison} reports the regret, communication, and action-set-oracle calls of the relevant methods under their respective access models. Because SO and LOO calls are distinct operations, the table compares asymptotic resource profiles rather than equal-cost operations or wall-clock performance. For a detailed discussion on oracle efficiency we refer to Remark~\ref{rmk:access-oracle} and the relevant projection-free OCO literature.

\paragraph{Contributions.}
We introduce \algname, a decentralized barrier-FTRL algorithm with a first-order dual-state communication layer for online upper-linearizable maximization under separation access, and prove $\widetilde O\bigl(\sqrt{T}\bigr)$ per-agent network-aggregated expected $\alpha$-regret. Over $T$ rounds, each agent uses $T$ neighbor-mixing operations, communicates one $d$-dimensional state per operation, and makes $O(T\log(\kappa T))$ calls to the action-set separation oracle. The number of \solvername{} updates is explicitly bounded for the controlled accuracy schedule. The general payoff theorem yields square-root-T $\alpha$-regret guarantees for four well known up-concave or DR-submodular maximization problems.

\paragraph{Technical Novelties.}
Our main challenges and technical novelties are as follow:
\begin{enumerate}

\item \textbf{A first-order communication interface for a curved local map.}
The barrier-FTRL map \eqref{eq:local-potential} has agent-dependent curvature, so directly coordinating Newton updates would require more than first-order messages (e.g., the Hessian). Our design instead communicates only the dual state through \eqref{eq:theta-update}, and the exact-average identity \eqref{eq:exact-average-state} makes this state sufficient to define the hypothetical network average problem, while every Hessian and Newton step remains local.

\item \textbf{Coupling row-wise network mixing with barrier-FTRL stability.}
The network-aggregate metric evaluates every payoff owner's function \(f_{t,j}\) at one specified agent \(a\)'s action, whereas the surrogate analysis initially controls each payoff at its owner's iterate. 
To bridge this gap, we retain the cumulative row-wise mixing coefficient \(M_a(W)\) (Lemma~\ref{lem:row-mixing}) and use it to control the specified agent's dual-state disagreement (Lemma~\ref{lem:dual-consensus}). 
Strong convexity of the barrier-FTRL potentials then converts this row-wise dual control into primal disagreement, improving the network dependence factor from $\sqrt N$ to $1+\log N$. 
The uniform stability of the composed payoffs $f_{t,j}\circ h$ converts primal disagreement into the owner-action cross-terms and prevents consensus errors from accumulating into an $O(T)$ payoff mismatch (Lemma~\ref{lem:minimizer-disagreement}, \ref{lem:feasible-point-disagreement} and \eqref{eq:fixed-comparator-aggregate-bound}).

\item \textbf{Controlled approximate local barrier-FTRL potential solver.}
Exact local minimization is unrealistic, but uncontrolled errors could accumulate linearly with the horizon. The \solvername{} subroutine uses the self-concordant stopping rule \eqref{eq:solver-stopping}, and its computable Newton-decrement test guarantees the Euclidean error bound (Lemma~\ref{lem:decrement-certificate}). The accuracy schedule in Lemma~\ref{lem:local-solve-accuracy} keeps the cumulative error bounded by $E_T\leq R$, while the warm-start analysis yields the finite iteration bound given by Proposition~\ref{prop:hybrid-newton-guarantee}.
% while the structured-Hessian formula in Lemma~\ref{lem:structured-hessian} makes every Newton direction an $O(d)$ computation.

\end{enumerate}

%======================================================================================================

\section{Preliminaries}
\label{sec:prelim}
\label{sec:preliminaries}

\subsection{Notations}

Let $\mathsf{G}=(\mathcal V,\mathcal E)$ denote an undirected communication
network of $N$ agents, where $\mathcal V=\{1,\ldots,N\}$ is the set of agents
and $\mathcal E$ is the set of communication links. We associate the network
with a mixing matrix $W\in\mathbb R^{N\times N}$, specified formally in
Assumption~\ref{ass:network}.

We denote the Euclidean norm by $\|\cdot\|_2$. For a positive definite matrix $\Sigma\in\mathbb R^{d\times d}$, the induced local norm is $\|x\|_{\Sigma}:=\sqrt{x^\top\Sigma x}$. We write $\mathbb B_2:=\{u\in\mathbb R^d:\|u\|_2\leq1\}$ for the closed Euclidean unit ball.

\subsection{Definitions}

\paragraph{Separation oracle access and gauge projection.}
Given a closed convex set $\mathcal K \subseteq\mathbb R^d$ and a query point $x\in\mathbb R^d$, a separation oracle (SO) $\operatorname{Sep}_{\mathcal K}(x)$ either certifies that $x\in\mathcal K$ or returns a nonzero normal vector $s\in\mathbb R^d$ such that $\langle s,x-z\rangle>0, \forall z \in\mathcal K$, and we call $s$ a hyperplane separating $x$ from $\mathcal K$.
We say an algorithm is under separation access when the algorithm can access the action set via a separation oracle, and when such separation oracle is indeed efficient over the action set. 
\begin{remark}[Two access oracle models]
\label{rmk:access-oracle}
    A separation oracle (SO) and a linear optimization oracle (LOO) expose complementary forms of action-set access, so their call counts do not by themselves determine computational cost. Separation is particularly natural for sets described by convex inequalities or intersections of simple constraints, where an infeasible query can be separated by returning a violated constraint and its normal, whereas linear optimization requires optimizing over the full intersection. For example, for a full-dimensional packing region $\mathcal K={x\in[0,1]^d\leq b}, A\geq0$, with a known strictly feasible point, an SO scans the box and resource constraints and returns a violated inequality, while an LOO solves the associated packing linear program. In matrix domains, separation over the spectral-norm ball requires only a leading singular-vector pair, whereas linear optimization generally requires a full-rank SVD; the ordering reverses for the nuclear-norm ball \citep{garber2022new}. Thus \emph{neither oracle model uniformly dominates the other}.
\end{remark}

Under Assumption~\ref{ass:geometry}, a separation oracle for $\mathcal K$ also gives one for the translated set $\mathcal C=\mathcal K-x_0$ by translation.
For the closed convex set $\mathcal C$ containing the origin in its interior, the Minkowski gauge function is $\gamma_{\mathcal C}(u):=\inf\{a>0:u\in a\mathcal C\},$ and gauge distance is defined as $S_{\mathcal C}(u):=\max\{0,\gamma_{\mathcal C}(u)-1\}$.
Using the gauge function and gauge distance, the exact gauge projection scales an infeasible point along the ray from the origin:
\begin{align*}
    p_{\mathcal C}(u):=\frac{u}{1+S_{\mathcal C}(u)}\in\mathcal C, \quad
    w_{\mathcal K}(u):=x_0+p_{\mathcal C}(u)\in\mathcal K.
\end{align*}
Thus a feasible point is unchanged, while an infeasible point is scaled to the boundary. 
Specifically, we use the SO-based gauge subroutine of \citet{mhammedi2025online}, \textsc{GaugeDist}, which uses an approximate gauge projection controlled by a precision parameter \(\varepsilon_{\mathrm{gau}}\), rather than the exact gauge projection. 
For completeness, we include the implementation along with useful properties and SO-complexity in Appendix~\ref{app:gauge-oracle}, and discuss how it interacts with our main algorithm in Section~\ref{sec:main-algorithm}.

\paragraph{Linearizable Framework and Uniform Wrappers.} 
\cite{pedramfar2024linear} first develops the linearizable framework in centralized setting, providing a clean reduction from online linearizable optimization to online linear optimization. \cite{pedramfar2025uniform} further refines the reduction process by formulating the uniform wrappers for action, query and function.
\cite{pedramfar2024linear} identifies three well known applications of the linearizable framework: monotone up-concave functions over general convex sets, monotone up-concave functions over convex sets containing the origin, and non-monotone up-concave functions over general convex sets. This is later extended to non-monotone DR-submodular maximization over down-closed convex sets with better approximation ratio and regret guarantee \citep{lu2026upper}. 
Moving beyond the centralized setting, \cite{lu2025decentralized} uses the linearizable framework as an engine for decentralized optimization of the first three classes.
Similarly, this work leverages it as a powerful interface so that as more function classes being added to the linearizable framework \citep{lu2026upper}, this will automatically create a compounding impact with our meta-algorithm to unlock online decentralized optimization of such functions. We formally introduce the definition of upper-linearizable functions.

\begin{definition}[Upper-linearizable function]
\label{def:upper-linearizable}
A function class $\mathcal F$ over $\mathcal K$ is upper-linearizable with structural parameters $(\alpha,\beta)\in(0,1]\times(0,\infty)$ and action map $h:\mathcal K\to\mathcal K$ 
if there exists an linearization proxy \(\mathfrak g:\mathcal F\times\mathcal K\to\mathbb R^d\)
such that, for every $f\in\mathcal F$ and $w,y\in\mathcal K$,
\begin{equation}
    \alpha f(y)-f(h(w))
    \leq
    \beta\langle\mathfrak g(f,w),y-w\rangle.
\end{equation}
\end{definition}

Following \citet{pedramfar2025uniform}, we associate the upper-linearizable representation and its base oracle with a fixed abstract uniform wrapper \(\mathcal W=(\mathcal W^{\mathrm{act}},\mathcal W^{\mathrm{qry}})\).
For each \(f\in\mathcal F\), let \(\mathcal O_f\) denote the base oracle\footnote{A base oracle may, for example, be a first-order oracle returning a possibly noisy gradient estimate, or a value oracle returning a possibly noisy function value.} through which \(f\) is accessed.
The action component is \(\mathcal W^{\mathrm{act}}=h\), while the query component transforms \(\mathcal O_f\) into a wrapped vector-valued oracle $\widehat{\mathcal O}_f := \mathcal W^{\mathrm{qry}}(\mathcal O_f).$
When invoked at an internal point \(w\), the wrapped oracle may use auxiliary randomness, query \(\mathcal O_f\) at one or more feasible points potentially different from both \(w\) and \(h(w)\), and transforms the resulting responses into a feedback vector. 
The wrapper pair associated with each function class and base oracle combination is specified in the applications.

\paragraph{Online Protocol.}
We consider an oblivious adversarial model. Before the game begins, the adversary fixes a sequence of local payoff functions \(\{f_{t,j}\in\mathcal F:t\in[T],\,j\in[N]\}\), where \(j\) indexes the owner of a local payoff. The future payoff functions and their oracle responses are not revealed to the agents in advance. 
At round \(t\), each agent \(a\in[N]\), using only the information available before the current feedback is observed, computes an internal feasible point \(w_{a,t}\in\mathcal K\) and plays \(\mathcal W^{\mathrm{act}}(w_{a,t})=h(w_{a,t}).\)
The agent then obtains access to the base oracle \(\mathcal O_{t,a}:=\mathcal O_{f_{t,a}}\), invokes the corresponding wrapped oracle at \(w_{a,t}\), and receives $q_{a,t}\sim\widehat{\mathcal O}_{t,a}(w_{a,t}),$ where $\widehat{\mathcal O}_{t,a} := \mathcal W^{\mathrm{qry}}(\mathcal O_{t,a}).$
The randomness in \(q_{a,t}\) may arise from the base oracle, the Query Wrapper, or both. Its required conditional mean and boundedness properties are given in Assumption~\ref{ass:feedback}.

\paragraph{Regret.} Under such protocol, the action \(x_{a,t}\) appearing in the regret definition from the Introduction is instantiated as $x_{a,t}=\mathcal W^{\mathrm{act}}(w_{a,t})=h(w_{a,t})$,
and the $\alpha$-regret\footnote{The coefficient $\alpha$ first appears as a structural parameter of the upper-linearization inequality, and it is not a tuning parameter of \algname{}. The payoff-lifting argument shows that this same coefficient scales the hindsight benchmark in the resulting regret guarantee, which motivates the common notation $R_\alpha^a(T)$.} of agent \(a\) becomes
\begin{align}
\mathfrak R_\alpha^a(T) := \mathbb E\!\left[
    \alpha\max_{y\in\mathcal K}\sum_{t=1}^T\mathsf F_t(y)
    -\sum_{t=1}^T\mathsf F_t\!\left(h(w_{a,t})\right) 
\right].
\label{eq:network-regret}
\end{align}
where $\mathsf F_t\!\left(\cdot\right) = \frac1N\sum_{j=1}^N f_{t,j}\!\left(\cdot\right),$ so the action-agent index \(a\) is held fixed while \(j\) ranges over all local-payoff owners, including \(j\neq a\).

\begin{remark}[Role of the approximation coefficient]
Even in the centralized offline setup, DR-submodular maximization could be NP-hard. As an example, \cite{bian17_guaran_non_optim} shows that for monotone DR-submodular functions $f$ over convex set containing the origin $\mathcal K$, it is NP-hard to find any point any point $x \in \mathcal{K}$ such that $f(x)$ is at least $(1 - e^{-1} + \epsilon)$ times the optimal value for any $\epsilon$. There are polynomial times algorithms that achieve the ratio of $1 - e^{-1}$, and such a number is referred to as the approximation coefficient, indicating the proportion of the optimal value we can achieve with polynomial time algorithms.
\end{remark}

\subsection{Assumptions}

\begin{assumption}[Translated geometry and separation access]
\label{assump:set}
\label{ass:geometry}
The feasible set $\mathcal K\subseteq\mathbb R^d$ is compact and convex. A
point $x_0\in\operatorname{int}(\mathcal K)$ and radii $0<r\leq R$ are known
such that, with $\mathcal C:=\mathcal K-x_0$,
\[
    r\mathbb B_2\subseteq\mathcal C\subseteq R\mathbb B_2.
\]
The set is accessed through a separation oracle for $\mathcal K$ and hence,
by translation, for $\mathcal C$. We set $\kappa:=R/r$.
\end{assumption}

\begin{assumption}[Network topology and communication matrix]
\label{assump:network}
\label{ass:network}
The communication network $\mathsf G=(\mathcal V,\mathcal E)$ is undirected and connected. Agents form convex combinations of their local states using a matrix $W\in\mathbb R^{N\times N}$ satisfying: 
1) Graph compatibility: $W_{ij}=0$ whenever $i\neq j$ and $\{i,j\}\notin\mathcal E$; 
2) Symmetry and double stochasticity: $W=W^\top$, $W_{ij}\geq0$, and $W\mathbf1=\mathbf1$; 
3) Contraction: with $J:=\mathbf1\mathbf1^\top/N$, $\rho:=\|W-J\|_2<1.$
Thus $\rho$ is the disagreement contraction factor and $1-\rho$ is the
absolute spectral gap.
\end{assumption}

% \begin{remark}[Interpretation of $\rho$ and the spectral gap]
% Let $J=\mathbf{1}\mathbf{1}^{\top}/N$. Since $Jz$ is the consensus component of any network state $z\in\mathbb{R}^N$ and $WJ=JW=J$, we have
% $\|Wz-Jz\|_2=\|(W-J)(z-Jz)\|_2\le \rho\|z-Jz\|_2$.
% Thus, $\rho=\|W-J\|_2$ controls the fraction of disagreement that can remain after one mixing step, motivating the name \emph{disagreement contraction factor}. Moreover, since $W$ is symmetric and doubly stochastic, $\lambda_1(W)=1$ and
% $\rho=\max_{i\ge 2}|\lambda_i(W)|$.
% Hence, $1-\rho=1-\max_{i\ge 2}|\lambda_i(W)|$ is the \emph{absolute spectral gap}: larger $1-\rho$ corresponds to faster contraction toward consensus, while smaller $1-\rho$ corresponds to slower network mixing.
% \end{remark}

\begin{assumption}[Conditionally unbiased bounded feedback]
\label{assump:gradients}
\label{ass:feedback}
Let $\mathcal F_t^{-}$ denote the sigma-field generated by the fixed payoff
sequence and all algorithmic and oracle randomness revealed before the
round-$t$ wrapped-oracle responses are generated, including the current
wrapper inputs $\{w_{a,t}\}_{a=1}^N$. For every agent $a$ and round $t$, the
wrapped feedback vector satisfies
\[
    \mathbb E[q_{a,t}\mid\mathcal F_t^{-}]
    =\mathfrak g(f_{t,a},w_{a,t}),
    \qquad
    \|q_{a,t}\|_2\leq G.
\]
\end{assumption}

\begin{assumption}[Uniform composed-payoff stability]
\label{ass:payoff-stability}
There is a horizon-independent constant $L_\phi\geq0$ such that, almost
surely, for every round $t$, payoff owner $j$, and $w,v\in\mathcal K$,
\begin{equation}
    \left|f_{t,j}(h(w))-f_{t,j}(h(v))\right|
    \leq L_\phi\|w-v\|_2.
    \label{eq:composed-payoff-lipschitz}
\end{equation}
Equivalently, each composed payoff $f_{t,j}\circ h$ is uniformly
$L_\phi$-Lipschitz on $\mathcal K$.
\end{assumption}

%======================================================================================================

\section{Decentralized Barrier Follow the Regularized Leader (Dec-BFTRL)}
\label{sec:main}

Having specified the online protocol and assumptions, we now present \algname{} and its main guarantee. The method combines three components: \textsc{GaugeDist} uses separation access to form a feasible wrapper input and the associated gauge correction; the agents convert their wrapped feedback into local surrogate losses and exchange only a single dual state; and \solvername{} locally computes an approximate minimizer of the resulting barrier-FTRL potential. 
Thus all feasibility and second-order computations remain local, while the network interaction is confined to one dual-state exchange per round.
Subsection~\ref{sec:main-algorithm} gives a round-by-round description of the method and its subroutines, while Subsection~\ref{sec:main-guarantee} states the main regret guarantee and provides a proof sketch. The supporting intermediate results and complete proof are deferred to Section~\ref{sec:analysis}.

\subsection{Main Algorithm}
\label{sec:main-algorithm}

At the beginning of round \(t\), each agent \(a\) holds two local states. The dual state \(\theta_{a,t}\in\mathbb R^d\) contains a mixed history of earlier surrogate losses and is the only state communicated to neighboring agents. The primal state \(u_{a,t}\in\operatorname{int}(R\mathbb B_2)\) is kept locally and approximately minimizes the \textit{BFTRL potential}
\begin{align}\label{eq:local-potential}
\Psi_{a,t}(u) := -\nu\log\!\left(R^2-\|u\|_2^2\right) +\frac{\mu_T}{2}\|u\|_2^2 +\langle\theta_{a,t},u\rangle .
\end{align}
and we denote its unique exact minimizer by $u_{a,t}^{\star} := \arg\min_{u\in\operatorname{int}(R\mathbb B_2)} \Psi_{a,t}(u).$
The barrier and quadratic terms together make \(\Psi_{a,t}\) strongly convex, with parameter $m_T:=\mu_T+\frac{2\nu}{R^2}.$
Initially, \(\theta_{a,1}=0\) and \(u_{a,1}=0=u_{a,1}^{\star}\). More generally, \((\theta_{a,t},u_{a,t})\) depends only on feedback from rounds preceding \(t\). Round \(t\) produces \((\theta_{a,t+1},u_{a,t+1})\), which is used at round \(t+1\). Algorithm~\ref{alg:dec-bftrl} summarizes this recursion.

\begin{algorithm}[h]
\caption{\algname: decentralized barrier-FTRL under separation access}
\label{alg:dec-bftrl}
\begin{algorithmic}[1]
\REQUIRE Horizon \(T\); separation-oracle access to \(\mathcal K\); known \(x_0,r,R\); symmetric doubly stochastic mixing matrix \(W\); abstract uniform wrapper \(\mathcal W=(\mathcal W^{\mathrm{act}},\mathcal W^{\mathrm{qry}})\); parameters \(\nu\geq1\), \(\mu_T>0\); local solve tolerances \(\{\varepsilon_{a,t}\}_{a\in[N],\,t=2,\ldots,T}\)
\STATE Set \(\mathcal C\leftarrow\mathcal K-x_0\) and \(\varepsilon_{\mathrm{gau}}\leftarrow 1/T\)
\STATE For every agent \(a\in[N]\), initialize \(\theta_{a,1}\leftarrow0\) and \(u_{a,1}\leftarrow0\)
\FOR{\(t=1,\ldots,T\)}
  \FOR{every agent \(a\in[N]\) in parallel}
    \STATE \((S_{a,t},s_{a,t})\leftarrow \textsc{GaugeDist} (\mathcal C,u_{a,t},\varepsilon_{\mathrm{gau}},r)\)
    \STATE \(p_{a,t}\leftarrow u_{a,t}/(1+S_{a,t})\) and \(w_{a,t}\leftarrow x_0+p_{a,t}\)
    \STATE Play \(x_{a,t}\leftarrow \mathcal W^{\mathrm{act}}(w_{a,t})=h(w_{a,t})\);
    \STATE Invoke the wrapped oracle \(\mathcal W^{\mathrm{qry}}(\mathcal O_{t,a})\) at \(w_{a,t}\) and receive \(q_{a,t}\)
    \STATE Let \(\ell_{a,t}\leftarrow-q_{a,t}\) and \(\displaystyle \widetilde\ell_{a,t}\leftarrow \ell_{a,t} - \mathbf 1\!\left\{\langle\ell_{a,t},u_{a,t}\rangle<0\right\} \langle\ell_{a,t},p_{a,t}\rangle s_{a,t}\)
    \STATE Exchange the dual state \(\theta_{a,t}\) with neighbors and update \(\displaystyle \theta_{a,t+1}\leftarrow \sum_{b=1}^N W_{ab}\theta_{b,t} +\widetilde\ell_{a,t}\)
    \IF{\(t<T\)}
      \STATE Form \(\Psi_{a,t+1}\) from \eqref{eq:local-potential}
      \STATE \(\displaystyle u_{a,t+1}\leftarrow \solvername( \Psi_{a,t+1}, u_{a,t}, \varepsilon_{a,t+1})\)
    \ENDIF
  \ENDFOR
\ENDFOR
\end{algorithmic}
\end{algorithm}

\paragraph{Gauge projection.}
The primal state \(u_{a,t}\) lies in the simple barrier domain \(R\mathbb B_2\), but it need not belong to the translated feasible set \(\mathcal C=\mathcal K-x_0\). The call to \textsc{GaugeDist} returns an approximate gauge distance \(S_{a,t}\) and an approximate subgradient \(s_{a,t}\), moderated by the gauge precision parameter \(\varepsilon_{\mathrm{gau}}\). Because \(S_{a,t}\) upper bounds the true gauge distance, the radial scaling $p_{a,t}=\frac{u_{a,t}}{1+S_{a,t}}$ belongs to \(\mathcal C\), and therefore \(w_{a,t}=x_0+p_{a,t}\) belongs to \(\mathcal K\). The vector \(s_{a,t}\) is retained for the surrogate-loss correction in the next step. This construction allows the barrier-FTRL update to operate on a Euclidean ball while producing a feasible input for the Action Wrapper. The corresponding feasibility, approximation, and separation-oracle guarantees are established in Section~\ref{sec:analysis-gauge}.

\paragraph{Feedback and surrogate loss.}
After constructing \(w_{a,t}\), the agent plays $x_{a,t} = \mathcal W^{\mathrm{act}}(w_{a,t}) = h(w_{a,t})$ and invokes the wrapped oracle \(\mathcal W^{\mathrm{qry}}(\mathcal O_{t,a})\) at \(w_{a,t}\), receiving the payoff-side vector \(q_{a,t}\). 
Because the barrier-FTRL update is written as a loss-minimization procedure, the agent first sets \(\ell_{a,t}=-q_{a,t}\). 
It then combines \(\ell_{a,t}\) with the gauge information \(p_{a,t}\) and \(s_{a,t}\) to form the corrected surrogate loss \(\widetilde\ell_{a,t}\), and the correction is activated only when \( \langle \ell_{a,t}, u_{a,t} \rangle < 0 \). 
Its role is to relate linear loss at the feasible point \(w_{a,t}\) to a bounded linear surrogate evaluated at the internal state \(u_{a,t}\). 
The precise comparison inequality and norm bound for \(\widetilde\ell_{a,t}\) are given in Section~\ref{sec:analysis-gauge}. The implementation and query cost of the wrapped oracle are specified separately for each application.

\paragraph{Communication and dual update.}
Once the current surrogate has been formed, each agent exchanges only its dual state \(\theta_{a,t}\) with its neighbors. Although this exchange appears after the feedback step in Algorithm~\ref{alg:dec-bftrl}, the communicated vector is the pre-update state and therefore contains only information from earlier rounds. Agent \(a\) first mixes these dual states and then injects its current local surrogate:
\begin{align}
\label{eq:theta-update}
    \theta_{a,t+1} = \sum_{b=1}^N W_{ab}\theta_{b,t} + \widetilde\ell_{a,t}.
\end{align}
Thus the feedback received at round \(t\) affects the primal state used at round \(t+1\), not the action already played at round \(t\). Graph compatibility of \(W\) makes the displayed sum implementable through neighbor communication. No primal iterate, feasible point, Hessian, or Newton information is exchanged. Section~\ref{sec:analysis-consensus} analyzes the evolution of the average dual state and controls disagreement through the network spectral gap.

\paragraph{HybridNewton update.}
The new dual state \(\theta_{a,t+1}\) defines the next local potential \(\Psi_{a,t+1}\). Agent \(a\) warm-starts \solvername{} from \(u_{a,t}\) and computes an interior point satisfying
\[
\|u_{a,t+1}-u_{a,t+1}^{\star}\|_2
\leq
\varepsilon_{a,t+1}.
\]
The logarithmic barrier keeps the Newton iterates inside \(R\mathbb B_2\), while strong convexity converts the solver’s Newton-decrement stopping condition into the stated Euclidean accuracy. The solver first uses damped Newton steps and then switches to full Newton steps once it enters the local convergence region. This computation is entirely local, and the returned point \(u_{a,t+1}\) begins the next round. The post-round solve is omitted when \(t=T\), since no subsequent action is needed. Section~\ref{sec:analysis-solver} establishes the accuracy guarantee and iteration complexity.
Each online round requires one exchange of a \(d\)-dimensional dual state per agent, and all gauge computations, wrapped-oracle operations, and HybridNewton updates remain local.

\subsection{Main Guarantee}
\label{sec:main-guarantee}

Let \(\delta:=1-\rho,\), \(\widetilde G:=2\kappa G\).
For main guarantee, we set \(\nu=1\), \(\mu_T=\frac{\widetilde G}{R} \sqrt{2T\left(1+\frac1\delta\right)}\), and \(m_T=\mu_T+\frac2{R^2}\).
We use gauge precision \(\varepsilon_{\mathrm{gau}}=1/T\) and local solve inaccuracy tolerances \(\varepsilon_{a,1}:=0\), and \(\varepsilon_{a,t}:=
\frac RT\) for \(2\leq t \leq T\).

For each agent \(a\in[N]\), define its row-mixing coefficient
\begin{equation}
    M_a(W) := \sum_{k=0}^{\infty}\sum_{b=1}^N \left| (W^k)_{ab}-\frac1N \right|.
    \label{eq:row-mixing-coefficient}
\end{equation}
Assumption~\ref{ass:network} ensures that \(M_a(W)<\infty\). Define also
\[
\mathcal B_T
:= \beta\left[
3GR+\widetilde G R+\log T
+\widetilde G R
\sqrt{2T\left(1+\frac1\delta\right)}
\right].
\]

\begin{theorem}[Dec-BFTRL: per-agent regret and total SO calls]
\label{thm:main}
Let $\mathsf F_t\!\left(\cdot\right) = \frac1N\sum_{j=1}^N f_{t,j}\!\left(\cdot\right)$.
Suppose every \(f_{t,j}\) belongs to the common \((\alpha,\beta)\)-upper-linearizable class of
Definition~\ref{def:upper-linearizable}, equipped with the fixed Action and Query Wrappers used by Algorithm~\ref{alg:dec-bftrl}. 
Suppose Assumptions~\ref{ass:geometry}, \ref{ass:network}, \ref{ass:feedback}, and \ref{ass:payoff-stability} hold with constants \(G>0\) and \(L_\phi\geq0\).
For every agent \(a\in[N]\), running Algorithm~\ref{alg:dec-bftrl} with the above parameters choices ensures
\begin{align}
\mathbb E\!\left[
    \alpha\max_{y\in\mathcal K}\sum_{t=1}^T\mathsf F_t(y)
    -\sum_{t=1}^T\mathsf F_t\!\left(h(w_{a,t})\right) 
\right]
\leq{}
\mathcal B_T
+2L_\phi R(2+\kappa)
+
L_\phi(1+\kappa)R
\sqrt{\frac{T\delta}{2(1+\delta)}}
\left(
M_a(W)+\frac1\delta
\right).
\label{eq:main-regret}
\end{align}
and each agent makes at most $T\left( 1+\left\lceil \log_2(4\kappa^2T) \right\rceil \right)$ total calls to the action-set separation oracle. 
This count excludes any action-set-oracle calls, if any, made by an
application-specific wrapper.
\end{theorem}

Section~\ref{sec:analysis-consensus} shows that $M_a(W) = O\!\left( \frac{1+\log N}{1-\rho} \right)$.
Therefore, Theorem~\ref{thm:main} implies for every agent \(a\)
\[
\mathfrak R_\alpha^a(T)
=
O\!\left(
\beta \log T + 
\left[
\beta\kappa GR
+
L_\phi(1+\kappa)R(1+\log N)
\right]
\sqrt{\frac{T}{1-\rho}}
\right)
= \widetilde O (\sqrt{T})
\]

\section{Analysis}
\label{sec:analysis}

Throughout this section, we work under the hypotheses and parameter choices of Theorem~\ref{thm:main}. All pathwise statements are understood on the probability-one event on which the almost-sure feedback bounds in Assumption~\ref{ass:feedback} hold simultaneously for all agents and rounds. In the Analysis, we extend the parameter setup and notations used in section~\ref{sec:main-guarantee}. We defer the detailed proof of each Lemma in this section to the Appendices.

\subsection{Gauge Projection and Surrogate Loss}
\label{sec:analysis-gauge}

Lemma~\ref{lem:gauge-surrogate} provides all the necessary information we need from the gauge projection subroutine, and its proof utilizes Lemma~\ref{lem:gauge-properties} and Lemma~\ref{lem:gauge-proj-prop} in Appendix~\ref{app:gauge-oracle}.

\begin{lemma}[Gauge projection and surrogate bound]
\label{lem:gauge-surrogate}
For every payoff owner \(j\) and round \(t\), the point constructed by Algorithm~\ref{alg:dec-bftrl} satisfies \(w_{j,t}\in\mathcal K\), and \(\|\widetilde\ell_{j,t}\|_2 \leq \widetilde G.\)
Moreover, for every \(v\in\mathcal C\),
\[\langle\ell_{j,t},w_{j,t}-(x_0+v)\rangle \leq \langle\widetilde\ell_{j,t},u_{j,t}-v\rangle +\frac{2GR}{T}.\]
\end{lemma}

\subsection{Network Mixing and Dual States}
\label{sec:analysis-consensus}

Recall from \eqref{eq:row-mixing-coefficient}, \(M_a(W) = \sum_{k=0}^{\infty}\sum_{b=1}^N \left| (W^k)_{ab}-\frac1N \right|.\)
Lemma~\ref{lem:row-mixing} is a useful technical lemma for the row mixing operation of the communication matrix and does not concern the dual state.
\begin{lemma}[Row mixing]
\label{lem:row-mixing}
For every agent \(a\in[N]\) and \(k\geq0\),
\begin{equation*}
    \sum_{b=1}^N
    \left|
    (W^k)_{ab}-\frac1N
    \right|
    \leq
    \min\{2,\sqrt N\,\rho^k\}.
    % \label{eq:row-mixing-decay}
\end{equation*}
Consequently, \(M_a(W)<\infty\). Moreover, \(M_a(W) = O\!\left(\frac{1+\log N}{1-\rho}\right).\)
\end{lemma}

At round $t$, define the hypothetical network-averaged dual state as 
\[ \bar\theta_t := \frac1N\sum_{j=1}^N\theta_{j,t}\] 
for \(1\leq t\leq T+1\). For \(1\leq t\leq T\) define the hypothetical network-averaged surrogate loss 
\[ \bar\ell_t := \frac1N \sum_{j=1}^N\widetilde\ell_{j,t}.\]
Also denote the norm of dual disagreement for agent $a$ as 
\[d_{a,t}:=\|\theta_{a,t}-\bar\theta_t\|_2.\]
Lemma~\ref{lem:dual-consensus} analyzes the disagreement in dual space, Lemma~\ref{lem:minimizer-disagreement} builds a bridge between primal space and dual space, and Lemma~\ref{lem:average-bftrl} provides the surrogate regret bound for hypothetical network average BFTRL primal problem.
\begin{lemma}[Dual-state averages and disagreement]
\label{lem:dual-consensus}
For every \(1\leq t\leq T+1\),
\begin{equation}
    \bar\theta_t=\sum_{s<t}\bar\ell_s.
    \label{eq:exact-average-state}
\end{equation}
Moreover,
\begin{equation}
    \left(
    \sum_{b=1}^N
    \|\theta_{b,t}-\bar\theta_t\|_2^2
    \right)^{1/2}
    \leq
    \frac{\sqrt N\,\widetilde G(1-\rho^{t-1})}{\delta}.
    \label{eq:theta-consensus}
\end{equation}
For every agent \(a\),
\begin{equation}
    \sum_{t=1}^T d_{a,t}
    \leq
    T\widetilde G M_a(W),
    \label{eq:row-state-sum}
\end{equation}
The agent-average disagreement also satisfies
\begin{equation}
    \sum_{t=1}^T\frac1N\sum_{b=1}^N d_{b,t}
    \leq
    \frac{T\widetilde G}{\delta}.
    \label{eq:average-state-sum}
\end{equation}
\end{lemma}

With the hypothetical network average dual state \(\bar\theta_t\), define the hypothetical network average potential
\begin{align}
    \bar\Psi_t(u)
    :=
    -\nu\log(R^2-\|u\|_2^2)
    +\frac{\mu_T}{2}\|u\|_2^2
    +\langle\bar\theta_t,u\rangle,
    \label{eq:average-potential}
\end{align}
and the minimizer of the average potential \(\bar u_t^\star := \arg\min_{u\in\operatorname{int}(R\mathbb B_2)} \bar\Psi_t(u).\)

\begin{lemma}[Minimizer disagreement]
\label{lem:minimizer-disagreement}
For every agent \(a\) and round \(t\), \(\|u_{a,t}^\star-\bar u_t^\star\|_2 \leq \frac{\|\theta_{a,t}-\bar\theta_t\|_2}{m_T}.\)
\end{lemma}

\begin{lemma}[Network-average BFTRL regret]
\label{lem:average-bftrl}
\(\forall v\in\operatorname{int}(R\mathbb B_2)\), \(\sum_{t=1}^T \langle\bar\ell_t,\bar u_t^\star-v\rangle
    \leq \bar\Psi_1(v)-\bar\Psi_1(0) +\frac{T\widetilde G^2}{m_T}.\)
\end{lemma}

\subsection{Local Solve Inaccuracy Tolerance}
\label{sec:analysis-solver}

Given the local solve inaccuracy tolerances \(\{\varepsilon_{a,t}\}_{a\in[N],\,t=2,\ldots,T}\), we define the total tolerance for agent $a$ over horizon $T$ as $E_{a,T}:=\sum_{t=1}^T\varepsilon_{a,t},$ and the average per-agent tolerance over horizon $T$ as $ E_T: = \frac1N \sum_{a=1}^N E_{a,T}.$

\begin{lemma}[Local solve tolerance]
\label{lem:local-solve-accuracy}
Given the tolerance schedule in Subsection~\ref{sec:main-guarantee}, every
call to \solvername{} terminates after finitely many Newton steps and returns
an interior point satisfying
\begin{equation}
    \|u_{a,t}-u_{a,t}^\star\|_2
    \leq
    \varepsilon_{a,t}.
    \label{eq:local-solve-accuracy}
\end{equation}
Moreover,
\begin{equation}
    E_{a,T}=\frac{(T-1)R}{T}<R,
    \qquad
    E_T< R.
    \label{eq:total-local-error}
\end{equation}
\end{lemma}

We define the common feasible point \(z_t:=x_0+p_{\mathcal C}(\bar u_t^\star)\), where \(p_{\mathcal C}\) is the exact gauge projection from Section~\ref{sec:prelim}, and \(\bar u_t^\star\) is the minimizer of hypothetical average potential \eqref{eq:average-potential}.

\begin{lemma}[Feasible-point disagreement]
\label{lem:feasible-point-disagreement}
For every agent \(a\) and round \(t\), the average internal primal state disagreement is bounded by 
\begin{align*}
    \sum_{t=1}^T\frac1N\sum_{j=1}^N
    \|w_{a,t}-w_{j,t}\|_2
    \leq{}
    2R+(1+\kappa)(E_{a,T}+E_T)
    +
    \frac{(1+\kappa)\widetilde G T}{m_T}
    \left(
    M_a(W)+\frac1\delta
    \right).
\end{align*}
\end{lemma}

\subsection{Proof of the Main Theorem}
\label{sec:analysis-main-proof}

\begin{proof}[Proof of Theorem~\ref{thm:main}]

Define the owner-local payoff term 
\[
    D_\alpha(T;y) := \mathbb E\left[ \frac1N\sum_{j=1}^N\sum_{t=1}^T \left(\alpha f_{t,j}(y) -f_{t,j}(h(w_{j,t})) \right) \right].
\]
Lemma~\ref{lem:gauge-surrogate} gives \(w_{j,t}\in\mathcal K\).
By the definition of \(\mathcal F_t^{-}\), \(w_{j,t}\) is \(\mathcal F_t^{-}\)-measurable, while \(f_{t,j}\) and \(y\) are fixed.
Moreover, \(q_{j,t}\), and hence \(\ell_{j,t}=-q_{j,t}\), is integrable by Assumption~\ref{ass:feedback}.
Therefore, conditional on \(\mathcal F_t^{-}\), Definition~\ref{def:upper-linearizable} and Assumption~\ref{ass:feedback} imply
\begin{align*}
    \alpha f_{t,j}(y)-f_{t,j}(h(w_{j,t}))
    &\leq
    \beta\langle
    \mathfrak g(f_{t,j},w_{j,t}),y-w_{j,t}
    \rangle\\
    &=
    \beta\,
    \mathbb E\left[
    \langle\ell_{j,t},w_{j,t}-y\rangle
    \mid\mathcal F_t^{-}
    \right].
\end{align*}
Taking total expectations yields 
\begin{align}
\label{eq:payoff-lifting}
D_\alpha(T;y)\leq \frac \beta N \sum_{j=1}^N\sum_{t=1}^T \langle \ell_{j,t},w_{j,t}-y\rangle,
\end{align}
and we call $\frac 1N\sum_{j=1}^N\sum_{t=1}^T \langle \ell_{j,t},w_{j,t}-y\rangle$ on the right hand side the linear surrogate regret.
Note that compared with the network aggregate regret, the owner-local payoff terms use the action for the agent of the payoff function, while the network regret fix the agent of the action. The gap between these two will be bounded in the second step after bounding the linear surrogate regret. After that, we will present parameter tuning to obtain best result, and analyze the total SO calls.

\textbf{Linear Surrogate Regret.}
Fix \(y\in\mathcal K\), let \(v_0:=y-x_0\in\mathcal C\), and set $v:=\left(1-\frac1T\right)v_0.$
Since $0\in\mathcal C$ and $\mathcal C$ is convex, $v\in\mathcal C$. Moreover, Assumption~\ref{ass:geometry} gives $\|v_0\|_2\le R$, and hence
\[
    \|v\|_2
    =
    \left(1-\frac{1}{T}\right)\|v_0\|_2
    \le
    \left(1-\frac{1}{T}\right)R
    < R.
\]
Thus $v\in\mathcal C\cap\operatorname{int}(R\mathbb B_2)$, so it is a valid comparator for the barrier-FTRL bound. 
Meanwhile,
\[
    \|x_0+v-y\|_2
    = \|-\frac{1}{T}(y-x_0)\|_2
    = \frac{1}{T}\|v_0\|_2
    \le
    \frac{R}{T}.
\]
Since \(\|\ell_{j,t}\|_2=\|q_{j,t}\|_2\leq G\) by Assumption~\ref{ass:feedback}, we have through Cauchy-Schwartz:
\[
    \frac1N\sum_{j=1}^N\sum_{t=1}^T  \langle\ell_{j,t},x_0+v-y\rangle
    \leq \frac1N\sum_{j=1}^N\sum_{t=1}^T \|\ell_{j,t}\|_2 \|x_0+v-y\|_2
    \leq GR.
\]
For each agent $j$ and round $t$, decompose $w_{j,t}-y=(w_{j,t}-(x_0+v)) + (x_0+v-y)$ we have
\begin{align*}
\langle\ell_{j,t},w_{j,t}-y\rangle
& = \langle \ell_{j,t},w_{j,t}-(x_0+v)\rangle 
+ \langle \ell_{j,t},x_0+v-y\rangle \\
& \leq \left( \langle \widetilde\ell_{j,t},u_{j,t}-v\rangle
    +\frac{2GR}{T} \right) + \|\ell_{j,t}\|_2\|x_0+v-y\|_2 \\
& \leq \langle \widetilde\ell_{j,t},u_{j,t}-v\rangle + \frac{3GR}{T},
\end{align*}
where the second inequality follows from Lemma~\ref{lem:gauge-surrogate} and cauchy-shawartz, and the third inequality follows from the bound on $ \|\ell_{j,t}\|_2$ and $\|x_0+v-y\|_2$. 
Summing over \(t\) and averaging over the $N$ agents then gives
\begin{equation}
    \frac1N\sum_{j=1}^N\sum_{t=1}^T
    \langle\ell_{j,t},w_{j,t}-y\rangle
    \leq
    3GR+
    \frac1N\sum_{j=1}^N\sum_{t=1}^T
    \langle\widetilde\ell_{j,t},u_{j,t}-v\rangle.
    \label{eq:main-reduction}
\end{equation}
The remaining term has the exact decomposition
\begin{align}
\frac1N\sum_{j=1}^N\sum_{t=1}^T \langle\widetilde\ell_{j,t},u_{j,t}-v\rangle
= \frac1N\sum_{j,t} \langle\widetilde\ell_{j,t}, u_{j,t}-u_{j,t}^\star\rangle
+ \frac1N\sum_{j,t} \langle\widetilde\ell_{j,t}, u_{j,t}^\star-\bar u_t^\star\rangle
+ \sum_{t=1}^T \langle\bar\ell_t,\bar u_t^\star-v\rangle.
\label{eq:three-way-decomposition}
\end{align}
The first term on the right is bounded by 
\begin{align*}
    \frac1N\sum_{j,t} \langle\widetilde\ell_{j,t}, u_{j,t}-u_{j,t}^\star\rangle
    \leq \frac1N \sum_{j,t} \|\widetilde\ell_{j,t}\|_2 \|u_{a,t}-u_{a,t}^\star\|_2 
    = \widetilde G \frac1N \sum_{j,t} \|u_{a,t}-u_{a,t}^\star\|_2 
    = \widetilde G E_T,
\end{align*}
where the first inequality is due to cauchy-shwarz inequality, the second equation comes from Lemmas~\ref{lem:gauge-surrogate} and the third equation comes from Lemma~\ref{lem:local-solve-accuracy}.

The second term is bounded by 
\begin{align*}
    \frac1N\sum_{j,t} \langle\widetilde\ell_{j,t}, u_{j,t}^\star-\bar u_t^\star\rangle
    \leq \frac1N\sum_{j,t} \|\widetilde\ell_{j,t}\|_2 \|u_{a,t}^\star-\bar u_t^\star\|_2
    \leq \frac{\widetilde G}{m_T} \sum_{t=1}^T\frac1N\sum_{j=1}^N d_{j,t}
    \leq \frac{T\widetilde G^2}{m_T\delta}.
\end{align*}
where the first inequality is due to cauchy-shwarz inequality, the second inequality comes from Lemma~\ref{lem:gauge-surrogate} and Lemma~\ref{lem:minimizer-disagreement}, the third and forth inequality come from Lemma~\ref{lem:dual-consensus}.

The third term is directed bounded by Lemma~\ref{lem:average-bftrl} with
\begin{align*}
    \sum_{t=1}^T \langle\bar\ell_t,\bar u_t^\star-v\rangle
    & \leq \bar\Psi_1(v)-\bar\Psi_1(0) +\frac{T\widetilde G^2}{m_T} \\
    &= -\nu\log\left( 1-\frac{\|v\|_2^2}{R^2} \right)
    +\frac{\mu_T}{2}\|v\|_2^2 + \frac{T\widetilde G^2}{m_T} \\
    &\leq \nu\log T+\frac{\mu_TR^2}{2} + \frac{T\widetilde G^2}{m_T}
\end{align*}
where the first inequality is from Lemma~\ref{lem:average-bftrl}, the second equation is because \eqref{eq:exact-average-state} gives \(\bar\theta_1=0\), and the third inequality is because \(\|v\|_2\leq(1-1/T)R\) and
\(1-(1-1/T)^2\geq1/T\).

Substitution of the three partial bounds back into
\eqref{eq:main-reduction} bounds the linear surrogate regret:
\begin{align}\label{eq:linear-regret}
    \frac1N\sum_{j=1}^N\sum_{t=1}^T
    \langle\ell_{j,t},w_{j,t}-y\rangle
    \leq
    B_{\mathrm{lin}}
    :={}
    3GR+\nu\log T+\frac{\mu_TR^2}{2}
    +\widetilde G E_T+
    \frac{T\widetilde G^2}{m_T}
    \left(1+\frac1\delta\right).
\end{align}
Thus, we have $D_\alpha(T;y)\leq\beta B_{\mathrm{lin}}$.

\textbf{Network-aggregate regret.}
For a fixed comparator \(y\in\mathcal K\), write
\[
    \mathfrak R_\alpha^a(T;y)
    :=
    \mathbb E\left[
    \sum_{t=1}^T
    \left(
    \alpha\mathsf F_t(y)
    -\mathsf F_t(h(w_{a,t}))
    \right)
    \right],
\]
where $\mathsf F_t\!(y) = \frac1N\sum_{j=1}^N f_{t,j}\!(y)$
and let \(\phi_{t,j}:=f_{t,j}\circ h\). Expanding the aggregate payoff gives the exact identity
\begin{align}
\mathfrak R_\alpha^a(T;y)
&= \mathbb E\left[ \sum_{t=1}^T
\left( \alpha\mathsf F_t(y) -\mathsf F_t(h(w_{a,t})) \right) \right] 
= \mathbb E\left[ \frac1N\sum_{j=1}^N\sum_{t=1}^T \left( \alpha f_{t,j}(y) -f_{t,j}(h(w_{a,t})) \right) \right] \nonumber\\
&= \mathbb E\left[ \frac1N\sum_{j=1}^N\sum_{t=1}^T \left( \alpha f_{t,j}(y) -f_{t,j}(h(w_{j,t})) \right) \right] 
+ \mathbb E\left[ \frac1N\sum_{j=1}^N\sum_{t=1}^T \left( f_{t,j}(h(w_{j,t})) -f_{t,j}(h(w_{a,t})) \right) \right] \nonumber\\
&= D_\alpha(T;y) + \mathbb E\left[ \sum_{t=1}^T\frac1N\sum_{j=1}^N \left( \phi_{t,j}(w_{j,t}) -\phi_{t,j}(w_{a,t}) \right) \right].
\label{eq:aggregate-transfer}
\end{align}

Recall Lemma~\ref{lem:gauge-surrogate} ensures that both \(w_{j,t}\) and \(w_{a,t}\) belong to \(\mathcal K\), so Assumption~\ref{ass:payoff-stability} applies. Since \(\phi_{t,j}=f_{t,j}\circ h\),
\[
\phi_{t,j}(w_{j,t})-\phi_{t,j}(w_{a,t})
\leq
\left|
\phi_{t,j}(w_{j,t})-\phi_{t,j}(w_{a,t})
\right|
\leq
L_\phi\|w_{j,t}-w_{a,t}\|_2.
\] 
Thus, substituting the linear surrogate regret of the local-owner payoffs $D_\alpha(T;y)\leq\beta B_{\mathrm{lin}}$ and the above inequality into \eqref{eq:aggregate-transfer} we have
\begin{align}
\mathfrak R_\alpha^a(T;y)
&\leq \beta B_{\mathrm{lin}} + L_\phi\, \mathbb E\left[ \sum_{t=1}^T\frac1N\sum_{j=1}^N \|w_{j,t}-w_{a,t}\|_2 \right] \nonumber\\
&\leq \beta B_{\mathrm{lin}} + L_\phi\left[ 2R+(1+\kappa)(E_{a,T}+E_T)  + \frac{(1+\kappa)\widetilde G T}{m_T}  \left( M_a(W)+\frac1\delta \right) \right].
\label{eq:fixed-comparator-aggregate-bound}
\end{align}
where the second inequality comes from Lemma~\ref{lem:feasible-point-disagreement}.
The payoff sequence is fixed before the learner's randomness, and the hindsight maximum is attained. We may therefore select \(y^\star\in \arg\max_{y\in\mathcal K} \sum_{t=1}^T \mathsf F_t(y)\), then \(\mathfrak R_\alpha^a(T) =\mathfrak R_\alpha^a(T;y^\star)\).

\textbf{Parameter choices and separation-oracle calls.}
Lemma~\ref{lem:local-solve-accuracy} gives
\(E_{a,T},E_T\leq R\). Since \(m_T\geq\mu_T\),
\[
\frac{\mu_TR^2}{2}
    + \frac{T\widetilde G^2}{m_T}
    \left(1+\frac1\delta\right)
    \leq
    \frac{\mu_TR^2}{2}
    +
    \frac{T\widetilde G^2}{\mu_T}
    \left(1+\frac1\delta\right)
    =
    \widetilde G R
    \sqrt{2T\left(1+\frac1\delta\right)}.
\]
Using \(\nu=1\), it follows that \(\beta B_{\mathrm{lin}}\leq\mathcal B_T.\)
Furthermore, \(\frac{\widetilde G T}{m_T} \leq \frac{\widetilde G T}{\mu_T} = R\sqrt{\frac{T\delta}{2(1+\delta)}}.\)
Substituting these bounds back to
\eqref{eq:fixed-comparator-aggregate-bound}, and observing that \(2R+(1+\kappa)(E_{a,T}+E_T) \leq2R(2+\kappa),\) we have \eqref{eq:main-regret}:
\[
\mathfrak R_\alpha^a(T)
\leq{}
\mathcal B_T
+2L_\phi R(2+\kappa)
+
L_\phi(1+\kappa)R
\sqrt{\frac{T\delta}{2(1+\delta)}}
\left(
M_a(W)+\frac1\delta
\right).
\]
where \(\mathcal B_T
=
\beta\left[
3GR+\widetilde G R+\log T
+\widetilde G R
\sqrt{2T\left(1+\frac1\delta\right)}
\right].\)

Finally, \(u_{a,1}=0\), and Lemma~\ref{lem:local-solve-accuracy} ensures that \(u_{a,t}\in\operatorname{int}(R\mathbb B_2)\) for \(2\leq t\leq T\). Thus every \textsc{GaugeDist} query satisfies the domain hypothesis of Lemma~\ref{lem:gauge-proj-prop}. Since \(\varepsilon_{\rm gau}=1/T\in(0,1]\), that lemma gives at most \(1+\left\lceil\log_2(4\kappa^2T)\right\rceil\) separation-oracle calls per invocation. There is one invocation per agent-round. By translation, each call to \(\operatorname{Sep}_{\mathcal C}\) is implemented by one call to \(\operatorname{Sep}_{\mathcal K}\), proving the stated per-agent \textsc{GaugeDist} total.
\end{proof}

%======================================================================================================

\section{Conclusion, Limitations, and Future Works}
\label{sec:conclusion}

We proposed \algname, for decentralized online upper-linearizable optimization under separation-oracle access. The algorithm attains agent-average expected approximate regret $\widetilde O\!\left(\sqrt{T}\right)$ with linear communication in horizon and $O(T\log(\kappa T))$ total separation-oracle calls. Our analysis avoids distributed curvature tracking. The agents mix a single first-order dual state, and minimizer sensitivity converts disagreement in that state into a primal regret term. A controlled approximate hybrid Newton procedure supplies summable local error without adding communication.

The guarantee in this work is specific to domains with a known full-dimensional interior point and economical separation access. It does not impose a universal computational ordering between separation and linear-optimization oracles. Extension to relative-interior geometries such as matroid basis polytopes remains open, and it is necessary before incorporating application to other classes such as the one-sided smooth functions.
Separately, for the applications to DR-submodular classes, this work only discusses static regret under the natural first-order feedback setting. Extensions to limited feedback setting, including zeroth-order and bandit feedback, and the investigation of non-stationary environment like dynamic and adaptive regret, remain open.

%======================================================================================================

\newpage
\bibliography{main}
\bibliographystyle{plainnat}

\newpage
\appendix
\section{Approximate Gauge Projection}
\label{app:gauge-oracle}

This appendix records the separation-based approximate gauge projection subroutine (Algorithm~\ref{alg:gauge-dist}) used by the main algorithm, whose construction and guarantees are illustrated by \citet{mhammedi2025online}, and no novelty is claimed for this subroutine.
We state it for the translated set $\mathcal C=\mathcal K-x_0$, for which $r\mathbb B_2\subseteq\mathcal C\subseteq R\mathbb B_2$. Lemma ~\ref{lem:gauge-properties} describes useful properties of the gauge function and gauge distance, and Lemma~\ref{lem:gauge-proj-prop} describes the feasibility guarantee and budget of calls to the SO.

\begin{algorithm}[H]
\caption{\textsc{GaugeDist}: approximate gauge distance and subgradient}
\label{alg:gauge-dist}
\begin{algorithmic}[1]
\REQUIRE Separation oracle $\operatorname{Sep}_{\mathcal C}$, query
$u\in\mathbb R^d$, precision
$\varepsilon_{\rm gau}\in(0,1]$, and inner radius $r$
\STATE $(b,v)\leftarrow\operatorname{Sep}_{\mathcal C}(u)$
\IF{$b=1$}
    \RETURN $(S,s)=(0,0)$
\ENDIF
\STATE $\alpha\leftarrow0$, $\beta\leftarrow1$,
$\zeta\leftarrow(\alpha+\beta)/2$
\WHILE{$\beta-\alpha>
 r^2\varepsilon_{\rm gau}/(2\|u\|_2^2)$}
    \STATE $(b,v_{\rm new})\leftarrow
    \operatorname{Sep}_{\mathcal C}(\zeta u)$
    \IF{$b=1$} 
    \STATE $\alpha\leftarrow\zeta$
    \ELSE
        \STATE $\beta\leftarrow\zeta$ and $v\leftarrow v_{\rm new}$
    \ENDIF
    \STATE $\zeta\leftarrow(\alpha+\beta)/2$
\ENDWHILE
\STATE $S\leftarrow\alpha^{-1}-1$ and
$s\leftarrow v/(\beta\langle v,u\rangle)$
\RETURN $(S,s)$
\end{algorithmic}
\end{algorithm}

As introduced in Preliminaries, this algorithm is built on gauge function and gauge distance 

\begin{lemma}[Gauge properties]
\label{lem:gauge-properties}
The function $\gamma_{\mathcal C}$ is convex and positively homogeneous.
Moreover,
\[
 \gamma_{\mathcal C}(u)\leq\frac{\|u\|_2}{r},
 \qquad
 S_{\mathcal C}(u)=\max\{0,\gamma_{\mathcal C}(u)-1\}.
\]
Every subgradient $s$ of $\gamma_{\mathcal C}$ has Euclidean norm at most
$1/r$.
\end{lemma}

\begin{proof}
These are the standard gauge properties collected in
\citet[Lemma~2.1 and Appendix~G]{mhammedi2025online}.  They apply to
$\mathcal C$ because the translated geometry places the origin in its
interior and contains $r\mathbb B_2$.
\end{proof}

\begin{lemma}[Approximate gauge guarantee and SO call budget]
\label{lem:gauge-proj-prop}
For any $u\in R\mathbb B_2$ and
precision $\varepsilon_{\rm gau}\in(0,1]$,
Algorithm~\ref{alg:gauge-dist} returns $(S,s)$ satisfying
\begin{align*}
    S_{\mathcal C}(u)
    \leq S\leq S_{\mathcal C}(u)+\varepsilon_{\rm gau},
    \quad \text{and} \quad \|s\|_2 \leq\frac1r,
    \quad \text{and}\quad S_{\mathcal C}(v)
    \geq S_{\mathcal C}(u)+\langle s,v-u\rangle
      -\varepsilon_{\rm gau},
    \forall v\in\mathbb R^d.
\end{align*}
If $u\in\mathcal C$, the algorithm terminates after one
separation-oracle call.  If $u\notin\mathcal C$, it makes at most
\begin{equation}
 1+\left\lceil
 \log_2\!\left(\frac{4\|u\|_2^2}
 {r^2\varepsilon_{\rm gau}}\right)
 \right\rceil
 \leq
 1+\left\lceil
 \log_2\!\left(\frac{4\kappa^2}
 {\varepsilon_{\rm gau}}\right)
 \right\rceil
\label{eq:gauge-call-budget}
\end{equation}
separation-oracle calls.  In particular, the right hand side expression is a
uniform upper bound in either case.
\end{lemma}

\begin{proof}
The three approximation properties and the exterior-query call count are
exactly \citet[Lemma~2.2]{mhammedi2025online}, applied to
$\mathcal C$.  An interior query returns immediately.  For an exterior
query, $\|u\|_2>r$, so the displayed logarithm is positive; its second
inequality follows from $\|u\|_2\leq R$ and $\kappa=R/r$.
\end{proof}

At the precision $\varepsilon_{\rm gau}=1/T$ used by \algname,
the per-agent budget over the horizon is therefore
\[
 T\left(1+\left\lceil\log_2(4\kappa^2T)\right\rceil\right)
 =O\!\left(T\log(\kappa T)\right).
\]
The network-wide budget is $N$ times this quantity.

%======================================================================================================

\section{HybridNewton: Controlled Approximate Local Barrier-FTRL Solver}
\label{app:hybrid-newton}

This appendix defines and analyzes \solvername{}, the local subroutine used by \algname{} to approximately minimize a local barrier-FTRL potential. The subroutine is a damped/full Newton implementation equipped with control over a computable Euclidean-accuracy. Its role here is to provide the local-solve guarantee required by the regret analysis, and it uses no communication, feedback-oracle calls, or action-set-oracle calls.

For a target dual vector $\theta\in\mathbb R^d$, regularization parameters $\mu>0$ and $\nu\geq1$, and radius $R>0$, define the generic barrier-FTRL potential for $u\in\operatorname{int}(R\mathbb B_2)$ by
\begin{equation}
 F(u):=-\nu\log(R^2-\|u\|_2^2)
 +\frac{\mu}{2}\|u\|_2^2+\langle\theta,u\rangle.
 \label{eq:generic-potential}
\end{equation}
Set \(m:=\mu+2\nu/R^2\). Then \(\nabla^2F(u)\succeq mI\). Since \(F\) diverges at the boundary, it has a unique interior minimizer
\[
 u^\star:=\arg\min_{u\in\operatorname{int}(R\mathbb B_2)}F(u).
\]

For an interior point $u$, define the Newton decrement of $F$ at $u$ by
\[
\lambda(u):= 
 \|\nabla F(u)\|_{[\nabla^2F(u)]^{-1}}.
\]
Given a prescribed Euclidean accuracy \(\varepsilon>0\), define the stopping
threshold
\begin{equation}
 \tau_\varepsilon:=
 \frac{\sqrt{m}\,\varepsilon}
 {1+\sqrt{m}\,\varepsilon}.
\label{eq:solver-stopping}
\end{equation}

We record the standard self-concordant facts used in the solver analysis. For $H(u):=\nabla^2F(u)$, write
\[
 \|h\|_u:=\sqrt{h^\top H(u)h},
 \qquad
 \|z\|_{u,*}:=\sqrt{z^\top H(u)^{-1}z}.
\]
Thus $\lambda(u)=\|\nabla F(u)\|_{u,*}$, and the unit Dikin ellipsoid at $u$ is given by
\[
 \mathcal E_u:=\{u+h:\|h\|_u<1\}.
\]
For $\nu\geq1$, the logarithmic term in \eqref{eq:generic-potential} is self-concordant on $\operatorname{int}(R\mathbb B_2)$. Adding the convex quadratic and linear terms preserves self-concordance on the same open domain, and the resulting $F$ still diverges at the boundary \citep{nemirovski2008interior}.

Standard self-concordant analysis gives $\mathcal E_u\subseteq\operatorname{dom}F$. If $\lambda(u)<1$, then
\begin{equation}
 \|u-u^\star\|_u
 \leq\frac{\lambda(u)}{1-\lambda(u)}.
 \label{eq:sc-proximity}
\end{equation}
Moreover, the damped update in \eqref{eq:damped-newton} satisfies
\begin{equation}
 F(u^+)\leq F(u)-\bigl(\lambda(u)-\log(1+\lambda(u))\bigr),
 \label{eq:sc-damped-decrease}
\end{equation}
and the full update in \eqref{eq:full-newton}, whenever $\lambda(u)<1$, satisfies
\begin{equation}
 \lambda(u^+)\leq
 \left(\frac{\lambda(u)}{1-\lambda(u)}\right)^2.
 \label{eq:sc-full-contraction}
\end{equation}
These are cited standard properties rather than new claims of this paper.
The computable stopping test $\lambda(u)\leq\tau_\varepsilon$ guarantees the requested Euclidean accuracy, as shown below.

\begin{algorithm}[H]
\caption{\solvername: Local Barrier-FTRL Solver with Guaranteed Accuracy}
\label{alg:hybrid-newton}
\begin{algorithmic}[1]
\REQUIRE Potential $F$ of the form \eqref{eq:generic-potential}, interior
initialization $u_0$, and accuracy $\varepsilon>0$
\ENSURE $u^+$ satisfying
$\lambda(u^+)\leq\tau_\varepsilon$ and hence
$\|u^+-u^{\star}\|_2
\leq\varepsilon$
\STATE $u\leftarrow u_0$
\STATE $\tau_\varepsilon\leftarrow
\sqrt{m}\varepsilon/
(1+\sqrt{m}\varepsilon)$
\WHILE{$\lambda(u)>1/4$}
  \STATE $u\leftarrow u-\frac{1}{1+\lambda(u)}
  [\nabla^2F(u)]^{-1}\nabla F(u)$
\ENDWHILE
\WHILE{$\lambda(u)>\tau_\varepsilon$}
  \STATE $u\leftarrow u-
  [\nabla^2F(u)]^{-1}\nabla F(u)$
\ENDWHILE
\STATE $u^+\leftarrow u$
\RETURN $u^+$
\end{algorithmic}
\end{algorithm}

The next two lemmas isolate the safety and accuracy facts used in the main analysis. Proposition~\ref{prop:hybrid-newton-guarantee} then summarizes finite termination and the guaranteed output of the complete subroutine.  The final lemma records the cost of one exact Newton iteration.

\begin{samepage}
\begin{lemma}[Interior safety]
\label{lem:dikin-safety}
Suppose $\nu\geq1$ and
$u\in\operatorname{int}(R\mathbb B_2)$.  The damped Newton update
\begin{equation}
 u^+=u-\frac{[\nabla^2F(u)]^{-1}\nabla F(u)}
 {1+\lambda(u)}
\label{eq:damped-newton}
\end{equation}
remains in $\operatorname{int}(R\mathbb B_2)$.  The full Newton update
\begin{equation}
 u^+=u-[\nabla^2F(u)]^{-1}\nabla F(u)
\label{eq:full-newton}
\end{equation}
also remains in the domain whenever $\lambda(u)<1$.
\end{lemma}
\end{samepage}

\begin{proof}
If $d$ is the damped displacement, then
\[
 \|d\|_{\nabla^2F(u)}
 =\frac{\lambda(u)}{1+\lambda(u)}<1.
\]
Hence $u+d\in\mathcal E_u\subseteq\operatorname{dom}F$.  For the full step,
the displacement has local norm $\lambda(u)<1$, and the same argument
applies.
\end{proof}

\begin{lemma}[Euclidean accuracy from the decrement]
\label{lem:decrement-certificate}
If $u$ is interior and $\lambda(u)\leq\delta<1$, then
\[
 \|u-u^{\star}\|_2
 \leq\frac{\delta}{(1-\delta)\sqrt{m}}.
\]
In particular, setting $\delta=\tau_\varepsilon$ from
\eqref{eq:solver-stopping} gives
\begin{equation}
 \|u^+-u^{\star}\|_2\leq\varepsilon.
\label{eq:solver-distance}
\end{equation}
\end{lemma}

\begin{proof}
The bound \eqref{eq:sc-proximity} and monotonicity of $x/(1-x)$ on
$[0,1)$ give
$\|u-u^{\star}\|_{\nabla^2F(u)}\leq\delta/(1-\delta)$.
Since $\nabla^2F(u)\succeq m I$, this implies the displayed
Euclidean bound.  Substitution of \eqref{eq:solver-stopping} gives
\eqref{eq:solver-distance}.
\end{proof}

\begin{proposition}[Guaranteed finite termination]
\label{prop:hybrid-newton-guarantee}
For every potential $F$ of the form \eqref{eq:generic-potential}, every
interior initialization $u_0$, and every accuracy $\varepsilon>0$,
Algorithm~\ref{alg:hybrid-newton} is well defined and terminates after
finitely many Newton updates. Every iterate remains in
$\operatorname{int}(R\mathbb B_2)$, and its output satisfies
\[
 \lambda(u^+)\leq\tau_\varepsilon,
 \qquad
 \|u^+-u^\star\|_2\leq\varepsilon.
\]
\end{proposition}

\begin{proof}
Let $\omega(s):=s-\log(1+s)$. During the first loop,
Lemma~\ref{lem:dikin-safety} makes every damped update well defined and
keeps the new iterate interior. By \eqref{eq:sc-damped-decrease}, each such
update with $\lambda(u)>1/4$ decreases $F$ by at least
$\omega(1/4)>0$. Since $F$ is bounded below by $F(u^\star)$, the first
loop terminates after finitely many updates with $\lambda(u)\leq1/4$.

If $\tau_\varepsilon\geq1/4$, the second loop is skipped. Otherwise, every
full update in that loop remains interior by Lemma~\ref{lem:dikin-safety},
and \eqref{eq:sc-full-contraction} yields
\[
 \lambda(u^+)\leq
 \left(\frac{\lambda(u)}{1-\lambda(u)}\right)^2
 \leq 2\lambda(u)^2
 \qquad\text{whenever }\lambda(u)\leq\frac14.
\]
Writing $b_k:=2\lambda(u_k)$ for the full-step iterates gives
$b_{k+1}\leq b_k^2$ and $b_0\leq1/2$, hence
$b_k\leq2^{-2^k}$. Therefore the positive threshold
$\tau_\varepsilon$ is reached after finitely many full steps. The stopping
test gives $\lambda(u^+)\leq\tau_\varepsilon$, and
Lemma~\ref{lem:decrement-certificate} gives
$\|u^+-u^\star\|_2\leq\varepsilon$.
\end{proof}

For the calls made by \algname{}, take
$F=\Psi_{a,t}$, $\nu=1$, $\mu=\mu_T$, $m=m_T$,
$\theta=\theta_{a,t}$, $u_0=u_{a,t-1}$, and $\varepsilon=R/T$ for
$a\in[N]$ and $t=2,\ldots,T$. The convention
$\varepsilon_{a,1}=0$ is only bookkeeping: no solver call is made at
$t=1$, where $u_{a,1}=0$ is the exact minimizer of the initial potential.
Thus each agent invokes \solvername{} exactly $T-1$ times.

\begin{lemma}[Structured Newton direction]
\label{lem:structured-hessian}
For $u\in\operatorname{int}(R\mathbb B_2)$, set
\[
 D(u):=R^2-\|u\|_2^2,
 a(u):=\mu+\frac{2\nu}{D(u)},
 b(u):=\frac{4\nu}{D(u)^2}.
\]
Then
\[
 \nabla F(u)=a(u)u+\theta,\qquad
 \nabla^2F(u)=a(u)I+b(u)uu^\top,
\]
and, for every $v\in\mathbb R^d$,
\[
 [\nabla^2F(u)]^{-1}v
 =\frac{v}{a(u)}
 -\frac{b(u)u\langle u,v\rangle}
 {a(u)\bigl(a(u)+b(u)\|u\|_2^2\bigr)}.
\]
Taking $v=\nabla F(u)$, both the Newton direction and
$\lambda(u)^2=\nabla F(u)^\top[\nabla^2F(u)]^{-1}\nabla F(u)$ use
$O(d)$ arithmetic and $O(d)$ working memory. Thus one exact
\solvername{} iteration has those costs.
\end{lemma}

\begin{proof}
Differentiate the potential in \eqref{eq:generic-potential}.  The inverse
formula is the Sherman-Morrison identity applied to
$a(u)I+b(u)uu^\top$.
\end{proof}

%================================================================

\section{Proof of Lemma~\ref{lem:gauge-surrogate}}
\begin{proof}[Proof of Lemma~\ref{lem:gauge-surrogate}]
Suppress \(j,t\), and write \(u,S,s,p,\ell\) for the corresponding
quantities. Lemma~\ref{lem:gauge-properties} and~\ref{lem:gauge-proj-prop} give, by positive homogeneity,
\[
    \gamma_{\mathcal C}(p)
    =\frac{\gamma_{\mathcal C}(u)}{1+S}
    \leq
    \frac{1+S_{\mathcal C}(u)}{1+S}
    \leq1.
\]
Thus \(p\in\mathcal C\), \(w=x_0+p\in\mathcal K\), and, by Assumption~\ref{ass:geometry}, \(\|p\|_2\leq R\). 
Moreover, \(\|\ell\|_2=\|q\|_2\leq G\) by Assumption~\ref{ass:feedback}. 
Let \(\chi:=\mathbf1\{\langle\ell,u\rangle<0\}.\)
Because \(p\) is a positive multiple of \(u\), \(-\chi\langle\ell,p\rangle\geq0\). 
Hence
\[
\begin{aligned}
    \|\widetilde\ell\|_2
    \leq
    \|\ell\|_2+
    |\langle\ell,p\rangle|\|s\|_2 
    \leq
    G+\frac{GR}{r}
    \leq2\kappa G,
\end{aligned}
\]
where \(\|s\|_2\leq1/r\) follows from Lemma~\ref{lem:gauge-proj-prop}, and \(r\leq R\) follows from
Assumption~\ref{ass:geometry}.

Define \(L(z):= \langle\ell,z\rangle -\chi\langle\ell,p\rangle S_{\mathcal C}(z).\)
Let \(v\in\mathcal C\). 
Lemma~\ref{lem:gauge-proj-prop} gives, with \(\epsilon_{\rm gau}=1/T\),
\[S_{\mathcal C}(u) - S_{\mathcal C}(v) \leq  \langle s,u-v\rangle +\frac 1T.\]
Thus,
\begin{align}
    L(u)-L(v) &= \langle l,u-v\rangle - \chi \langle l,p\rangle \left(S_{\mathcal C}(u) - S_{\mathcal C}(v)\right) \nonumber \\
    & \leq \langle l,u-v\rangle - \chi \langle l,p\rangle \left(\langle s,u-v\rangle +\frac 1T \right) \nonumber \\
    & = \langle l- \chi \langle l,p\rangle s,u-v\rangle - \frac {\chi \langle l,p\rangle}T \nonumber \\
    & = \langle \widetilde l,u-v\rangle + \frac {|\langle l,p\rangle | }T,
    \label{eq:gauge-surrogate-upper}
\end{align}
where the last equality is due to \(-\chi\langle\ell,p\rangle = | \langle\ell,p\rangle |\) as the result of \(-\chi\langle\ell,p\rangle\geq0\).
% Since \(S_{\mathcal C}(v)=0\) for \(v\in\mathcal C\), the approximate
% subgradient inequality for $s$ in Lemma~\ref{lem:gauge-proj-prop} implies
% \begin{equation}
%     L(u)-L(v)
%     \leq
%     \langle\widetilde\ell,u-v\rangle
%     +\frac{|\langle\ell,p\rangle|}{T}.
%     \label{eq:gauge-surrogate-upper}
% \end{equation}

Separately, we also have 
\begin{equation}
    L(u)\geq\langle\ell,p\rangle-\frac{GR}{T}.
    \label{eq:gauge-surrogate-lower}
\end{equation}
because if \(S=0\), then \(p=u\), and Lemma~\ref{lem:gauge-proj-prop} gives \( S_{\mathcal C}(u) = 0 \) and \( L(u)=\langle l, u\rangle = \langle l, p\rangle \). If \(S>0\) and \(\chi=0\), then \(\langle\ell,u\rangle\geq0\) and \(\langle\ell,p\rangle\leq L(u)\). If \(S>0\) and \(\chi=1\), then \(L(u)-\langle\ell,p\rangle = \langle\ell,p\rangle \bigl(S-S_{\mathcal C}(u)\bigr) \geq-\frac{GR}{T}.\)

For \(v\in\mathcal C\), we have \(S_{\mathcal C}(v)=0\), and thus \(L(v)=\langle l, v \rangle.\) Therefore,
combining \eqref{eq:gauge-surrogate-upper} and \eqref{eq:gauge-surrogate-lower}, and using \(w=x_0+p\), we obtain
\begin{align*}
\langle \ell, w-(x_0+v)\rangle
&= \langle \ell,p-v\rangle 
= \langle \ell,p\rangle-\langle \ell,v\rangle \notag\\
&\le L(u)-L(v)+\frac{GR}{T} \notag\\
&\le \langle \widetilde{\ell},u-v\rangle +\frac{|\langle \ell,p\rangle|}{T} +\frac{GR}{T} \notag\\
&\le \langle \widetilde{\ell},u-v\rangle+\frac{2GR}{T}.
\end{align*}
where the first inequality comes from \eqref{eq:gauge-surrogate-lower}, the second inequality comes from \eqref{eq:gauge-surrogate-upper}, and the last inequality uses Cauchy-Schwarz and the fact that $p\in C\subseteq RB_2$ and
$\|\ell\|_2\le G$. 
This yields the final bound in Lemma~\ref{lem:gauge-surrogate}.
\end{proof}

%================================================================

\section{Proof of Lemma~\ref{lem:row-mixing}}
\begin{proof}[Proof of Lemma~\ref{lem:row-mixing}]
Recall \(J=\mathbf1\mathbf1^\top/N\).
By Assumption~\ref{ass:network}, the \(a\)-th rows of \(W^k\) and \(J\) are probability vectors, so their \(\ell_1\)-distance is at most \(2\).
Also, by Cauchy--Schwarz,
\[
    \sum_{b=1}^N
    \left|
    (W^k)_{ab}-\frac1N
    \right|
    \leq
    \sqrt N
    \left(
    \sum_{b=1}^N
    \left|
    (W^k)_{ab}-\frac1N
    \right|^2
    \right)^{1/2}
    \leq
    \sqrt N\,\|W^k-J\|_2.
\] 
For \(k=0\),
\(\|W^0-J\|_2\leq1\). For \(k\geq1\), double stochasticity gives
\(WJ=JW=J\), and hence
\[
    W^k-J=(W-J)^k,
    \qquad
    \|W^k-J\|_2\leq\rho^k.
\]
Combining the two conditions of $k$,
\[
\sum_{b=1}^N
    \left|
    (W^k)_{ab}-\frac1N
    \right|
    \leq
    \min\{2,\sqrt N\,\rho^k\}
\]

For \(0<\rho<1\), split the series in
\eqref{eq:row-mixing-coefficient} at \(k_0:=
    \left\lceil
    \frac{[\log(\sqrt N/2)]_+}{-\log\rho}
    \right\rceil\) to obtain
\[
    M_a(W)
    \leq
    2k_0+
    \sqrt N\sum_{k=k_0}^{\infty}\rho^k
    \leq
    2k_0+\frac2{1-\rho}.
\]
If \(\rho=0\), then \(W=J\), so only \(k=0\) contributes and
\[
    M_a(W)
    =
    \sum_{b=1}^N
    \left|
    \mathbf 1\{a=b\}-\frac1N
    \right|
    =
    2\left(1-\frac1N\right).
\]
Finally, \(-\log\rho\geq1-\rho\) for \(0<\rho<1\), which gives the
stated order bound together with the separate \(\rho=0\) case.
\end{proof}

%================================================================

\section{Proof of Lemma~\ref{lem:dual-consensus}}
\begin{proof}[Proof of Lemma~\ref{lem:dual-consensus}]

Double stochasticity in Assumption~\ref{ass:network} and
\eqref{eq:theta-update} give
\(\bar\theta_{t+1}=\bar\theta_t+\bar\ell_t\), and given the zero initialization, by recursion we have \eqref{eq:exact-average-state}.

Stack $\theta_{b,t}^{\top}$ as the rows of $\boldsymbol{\Theta}_t$, and stack $\widetilde{\ell}_{b,t}^{\top}$ as the rows of $\boldsymbol{L}_t$. The agent-wise update in \eqref{eq:theta-update} is equivalently
\[
    \boldsymbol{\Theta}_{t+1}
    =W\boldsymbol{\Theta}_t+\boldsymbol{L}_t.
\]
Moreover, $J\boldsymbol{\Theta}_t$ has every row equal to $\bar{\theta}_t^{\top}$, and hence
\[
    \|(I-J)\boldsymbol{\Theta}_t\|_{\mathrm F} = \left( \sum_{b=1}^N  \|\theta_{b,t} -\bar{\theta}_t\|_2^2 \right)^{1/2}.
\]
Since $JW=WJ=J$ and $J^2=J$, we have
\begin{align*}
    (I-J)\boldsymbol{\Theta}_{t+1}
    &=(I-J)W\boldsymbol{\Theta}_t+(I-J)\boldsymbol{L}_t  \\
    &=(W-J)(I-J)\boldsymbol{\Theta}_t  +    (I-J)\boldsymbol{L}_t.
\end{align*}
Therefore, using $\|W-J\|_2=\rho$, $\|I-J\|_2=1$, and since Lemma~\ref{lem:gauge-surrogate} gives
$\|\boldsymbol{L}_t\|_{\mathrm F}\le \sqrt{N}\,\widetilde G$,
\[
    \|(I-J)\boldsymbol{\Theta}_{t+1}\|_{\mathrm F}
    \le
    \rho\|(I-J)\boldsymbol{\Theta}_t\|_{\mathrm F}
    +\sqrt{N}\,\widetilde G.
\]
Since $\boldsymbol{\Theta}_1=0$, unrolling this recursion gives
\begin{align*}
    \|(I-J)\boldsymbol{\Theta}_t\|_{\mathrm F}
    &\le \sqrt{N}\,\widetilde G \sum_{k=0}^{t-2}\rho^k\\
    &= \frac{\sqrt{N}\,\widetilde G(1-\rho^{t-1})}{1-\rho}
    = \frac{\sqrt{N}\,\widetilde G(1-\rho^{t-1})}{\delta},
\end{align*}
which proves \eqref{eq:theta-consensus}.

Similarly, unrolling $\boldsymbol{\Theta}_{t+1} = W\boldsymbol{\Theta}_t + \boldsymbol{L}_t$ from $\boldsymbol{\Theta}_1=0$ gives
\[
    \boldsymbol{\Theta}_t = \sum_{s=1}^{t-1} W^{t-1-s}\boldsymbol{L}_s.
\]
Since $JW^k=J$ for every $k\ge 0$,
\[
    J\boldsymbol{\Theta}_t
    = \sum_{s=1}^{t-1} J\boldsymbol{L}_s.
\]
The $a$-th row of $\boldsymbol{\Theta}_t-J\boldsymbol{\Theta}_t$
therefore yields
\[
    \theta_{a,t}-\bar{\theta}_t
    =
    \sum_{s=1}^{t-1}\sum_{b=1}^N
    \left(
        (W^{t-1-s})_{ab}-\frac1N
    \right)
    \widetilde{\boldsymbol{\ell}}_{b,s}.
\]
Using $\|\widetilde{\boldsymbol{\ell}}_{b,s}\|_2
\le \widetilde G$, we obtain
\begin{align*}
\sum_{t=1}^T d_{a,t}
= \sum_{t=1}^T \|\theta_{a,t} -\bar{\theta}_t\|_2
&\le
\widetilde G
\sum_{t=2}^T
\sum_{s=1}^{t-1}
\sum_{b=1}^N
\left|
    (W^{t-1-s})_{ab}-\frac1N
\right|\\
&=
\widetilde G
\sum_{t=2}^T
\sum_{k=0}^{t-2}
\sum_{b=1}^N
\left|
    (W^k)_{ab}-\frac1N
\right|\\
&=
\widetilde G
\sum_{k=0}^{T-2}
(T-1-k)
\sum_{b=1}^N
\left|
    (W^k)_{ab}-\frac1N
\right|\\
&\le
T\widetilde G
\sum_{k=0}^{\infty}
\sum_{b=1}^N
\left|
    (W^k)_{ab}-\frac1N
\right|\\
&=
T\widetilde G\,M_a(W).
\end{align*}
where the second equality uses the change of variables
$k=t-1-s$, and the third equality changes the order of summation:
for each fixed $k\in\{0,\ldots,T-2\}$, the corresponding term
appears for $t=k+2,\ldots,T$, hence exactly $T-1-k$ times.
The last equality follows from the definition of $M_a(W)$.
This  proves \eqref{eq:row-state-sum}.

By Cauchy-Schwarz,
\begin{align*}
\frac{1}{N}\sum_{b=1}^N d_{b,t}
&\le \frac{1}{N}\sqrt{N} \left(\sum_{b=1}^N d_{b,t}^2\right)^{1/2}\\
&= \frac{1}{\sqrt{N}} \left(\sum_{b=1}^N \|\boldsymbol{\theta}_{b,t} -\bar{\boldsymbol{\theta}}_t\|_2^2\right)^{1/2}\\
&\le \frac{\widetilde G(1-\rho^{t-1})}{\delta},
\end{align*}
where the last inequality follows from \eqref{eq:theta-consensus}.
Summing over $t=1,\ldots,T$ therefore gives
\begin{align*}
\sum_{t=1}^T\frac{1}{N}\sum_{b=1}^N d_{b,t}
&\le \frac{\widetilde G}{\delta} \sum_{t=1}^T(1-\rho^{t-1})\\
&\le \frac{T\widetilde G}{\delta},
\end{align*}
where the last inequality uses $0\le 1-\rho^{t-1}\le 1$. This proves \eqref{eq:average-state-sum}.
\end{proof}

%================================================================

\section{Proof of Lemma~\ref{lem:minimizer-disagreement}}

\begin{proof}[Proof of Lemma~\ref{lem:minimizer-disagreement}]
The local and network-average potentials differ only in their linear
terms. In particular, writing
\(
    \Phi_T(u):=
    -\nu\log\!\left(R^2-\|u\|_2^2\right)
    +\frac{\mu_T}{2}\|u\|_2^2,
\)
we have
\[
    \Psi_{a,t}(u)
    =
    \Phi_T(u)+\langle\boldsymbol{\theta}_{a,t},u\rangle,
    \qquad
    \bar{\Psi}_t(u)
    =
    \Phi_T(u)+\langle\bar{\boldsymbol{\theta}}_t,u\rangle.
\]
Both potentials satisfy
$\nabla^2\Psi_{a,t}(u)
=\nabla^2\bar{\Psi}_t(u)\succeq m_T I$.
Hence, $\nabla\Psi_{a,t}$ is $m_T$-strongly monotone, and therefore
\begin{align*}
m_T
\|\bar{u}_t^\star-u_{a,t}^\star\|_2^2
&\le
\left\langle
\nabla\Psi_{a,t}(\bar{u}_t^\star)
-\nabla\Psi_{a,t}(u_{a,t}^\star),
\bar{u}_t^\star-u_{a,t}^\star
\right\rangle.
\end{align*}
Since both minimizers are interior, their respective first-order
conditions give
\[
    \nabla\Psi_{a,t}(u_{a,t}^\star)=0,
    \qquad
    \nabla\bar{\Psi}_t(\bar{u}_t^\star)=0.
\]
Moreover, because the two potentials differ only in their linear terms,
\[
    \nabla\Psi_{a,t}(\bar{u}_t^\star)
    =
    \nabla\bar{\Psi}_t(\bar{u}_t^\star)
    +\boldsymbol{\theta}_{a,t}
    -\bar{\boldsymbol{\theta}}_t
    =
    \boldsymbol{\theta}_{a,t}
    -\bar{\boldsymbol{\theta}}_t.
\]
By Cauchy-Schwarz,
\begin{align*}
m_T
\|\bar{u}_t^\star-u_{a,t}^\star\|_2^2
&\le
\left\langle
\boldsymbol{\theta}_{a,t}
-\bar{\boldsymbol{\theta}}_t,
\bar{u}_t^\star-u_{a,t}^\star
\right\rangle\\
&\le
\|\boldsymbol{\theta}_{a,t}
-\bar{\boldsymbol{\theta}}_t\|_2
\|\bar{u}_t^\star-u_{a,t}^\star\|_2.
\end{align*}
If $\bar{u}_t^\star=u_{a,t}^\star$, the claim is immediate;
otherwise, dividing by
$m_T\|\bar{u}_t^\star-u_{a,t}^\star\|_2$ gives
\[
    \|u_{a,t}^\star-\bar{u}_t^\star\|_2
    \le
    \frac{\|\boldsymbol{\theta}_{a,t}
    -\bar{\boldsymbol{\theta}}_t\|_2}{m_T}
    =
    \frac{d_{a,t}}{m_T}.
\]
\end{proof}

\section{Proof of Lemma~\ref{lem:average-bftrl}}

\begin{proof}[Proof of Lemma~\ref{lem:average-bftrl}]
Equation~\eqref{eq:exact-average-state} gives \(\bar\theta_t = \sum_{s=1}^{t-1}\bar\ell_s,\) and in particular $\bar\theta_1=0$. Therefore,
by \eqref{eq:average-potential},
\[
    \bar\Psi_t(u)
    =
    \bar\Psi_1(u)
    +
    \sum_{s=1}^{t-1}\langle\bar\ell_s,u\rangle.
\]
Hence,
\[
    \bar u_t^\star
    =
    \arg\min_{u\in\operatorname{int}(R\mathbb B_2)}
    \left\{
        \bar\Psi_1(u)
        +
        \sum_{s=1}^{t-1}\langle\bar\ell_s,u\rangle
    \right\},
\]
so $\bar u_t^\star$ is precisely the FTRL iterate generated by the
regularizer $\bar\Psi_1$ and the past vectors
$\bar\ell_1,\ldots,\bar\ell_{t-1}$. Since
$\bar\theta_1=0$, the symmetry of $\bar\Psi_1$ also gives
$\bar u_1^\star=0$.

The standard be-the-leader inequality for FTRL \citep{hazan2016introduction} states that, for $f_0,\ldots,f_T$ and $x_{t+1}\in\arg\min_x\sum_{s=0}^t f_s(x)$, we have 
$\sum_{t=0}^T f_t(x_{t+1}) \le \sum_{t=0}^T f_t(x)$ for every $x$. 
Applying it here with $f_0=\bar\Psi_1$ and $f_t(u)=\langle\bar\ell_t,u\rangle$ for $t\ge 1$, and using $\bar u_1^\star=0$, gives
\begin{equation}
    \sum_{t=1}^T \langle\bar\ell_t,\bar u_{t+1}^\star-v\rangle
    \le \bar\Psi_1(v)-\bar\Psi_1(\bar u_1^\star)
    = \bar\Psi_1(v)-\bar\Psi_1(0).
    \label{eq:btl-average}
\end{equation}

It remains to bound the difference between two consecutive FTRL minimizers. Since
\[
    \bar\Psi_{t+1}(u)
    =
    \bar\Psi_t(u)+\langle\bar\ell_t,u\rangle,
\]
their gradients satisfy
\[
    \nabla\bar\Psi_{t+1}(u)
    =
    \nabla\bar\Psi_t(u)+\bar\ell_t.
\]
The first-order conditions at the two minimizers therefore give
\[
    \nabla\bar\Psi_t(\bar u_t^\star)=0,
    \qquad
    \nabla\bar\Psi_t(\bar u_{t+1}^\star)
    =
    -\bar\ell_t.
\]
Since $\bar\Psi_t$ is $m_T$-strongly convex,
$\nabla\bar\Psi_t$ is $m_T$-strongly monotone. Thus,
\begin{align*}
m_T\|\bar u_t^\star-\bar u_{t+1}^\star\|_2^2
&\le
\left\langle
    \nabla\bar\Psi_t(\bar u_t^\star)
    -
    \nabla\bar\Psi_t(\bar u_{t+1}^\star),
    \bar u_t^\star-\bar u_{t+1}^\star
\right\rangle\\
&=
\left\langle
    \bar\ell_t,
    \bar u_t^\star-\bar u_{t+1}^\star
\right\rangle\\
&\le
\|\bar\ell_t\|_2
\|\bar u_t^\star-\bar u_{t+1}^\star\|_2.
\end{align*}
Therefore,
\[
    \|\bar u_t^\star-\bar u_{t+1}^\star\|_2
    \le
    \frac{\|\bar\ell_t\|_2}{m_T}.
\]
Moreover, Lemma~\ref{lem:gauge-surrogate} gives
\[
    \|\bar\ell_t\|_2
    =
    \left\|
        \frac1N\sum_{j=1}^N\widetilde\ell_{j,t}
    \right\|_2
    \le
    \frac1N\sum_{j=1}^N
    \|\widetilde\ell_{j,t}\|_2
    \le
    \widetilde G.
\]
Consequently,
\begin{equation}
    \left\langle
        \bar\ell_t,
        \bar u_t^\star-\bar u_{t+1}^\star
    \right\rangle
    \le
    \frac{\|\bar\ell_t\|_2^2}{m_T}
    \le
    \frac{\widetilde G^2}{m_T}.
    \label{eq:average-ftrl-stability}
\end{equation}

Finally, combining \eqref{eq:btl-average} and \eqref{eq:average-ftrl-stability} we have
\begin{align*}
\sum_{t=1}^T \langle\bar\ell_t,\bar u_t^\star-v\rangle
&= \sum_{t=1}^T \langle\bar\ell_t,\bar u_{t+1}^\star-v\rangle
+ \sum_{t=1}^T \langle\bar\ell_t, \bar u_t^\star-\bar u_{t+1}^\star\rangle\\
&\le \bar\Psi_1(v)-\bar\Psi_1(0)
+ \frac{T\widetilde G^2}{m_T},
\end{align*}
which proves the claim of Lemma~\ref{lem:average-bftrl}.
\end{proof}

\section{Proof of Lemma~\ref{lem:local-solve-accuracy}}
\begin{proof}[Proof of Lemma~\ref{lem:local-solve-accuracy}]

At \(t=1\), \(\theta_{a,1}=0\) and \(u_{a,1}=0=u_{a,1}^\star\). Consider a call at \(t\geq2\), and instantiate Algorithm~\ref{alg:hybrid-newton} in Appendix~\ref{app:hybrid-newton} with \(F=\Psi_{a,t}\), \(\nu=1\), \(\mu=\mu_T\), \(m=m_T\), \(\theta=\theta_{a,t}\), and \(\varepsilon=\varepsilon_{a,t}\). Write \(\lambda\) for the corresponding Newton decrement. Its stopping threshold
\eqref{eq:solver-stopping} is
\[
    \tau_{a,t}
    :=
    \frac{\sqrt{m_T}\,\varepsilon_{a,t}}
    {1+\sqrt{m_T}\,\varepsilon_{a,t}}
    \in(0,1).
\]
By induction, the warm start \(u_{a,t-1}\) is interior. While \(\lambda>1/4\), Lemma~\ref{lem:dikin-safety} keeps each damped step interior, while \eqref{eq:sc-damped-decrease}, together with monotonicity of \(\omega(s):=s-\log(1+s)\), decreases the fixed potential \(\Psi_{a,t}\) by at least \(\omega(1/4)>0\). Since \(\Psi_{a,t}\) is bounded below by its minimum, this phase is finite. Once \(\lambda\leq1/4\), Lemma~\ref{lem:dikin-safety} keeps every full step interior, and \eqref{eq:sc-full-contraction} gives
\[
    \lambda^+
    \leq
    \left(\frac{\lambda}{1-\lambda}\right)^2
    \leq2\lambda^2.
\]
If \(\tau_{a,t}\geq1/4\), the stopping condition already holds. Otherwise, for the  full-step decrements set \(b_k:=2\lambda_k\). Then \(b_{k+1}\leq b_k^2\) and \(b_0\leq1/2\), so \(b_k\leq2^{-2^k}\). Hence \(\lambda_k\leq\tau_{a,t}\) after finitely many steps. Applying Lemma~\ref{lem:decrement-certificate} with \(m=m_T\) and its threshold variable \(\delta=\tau_{a,t}\) yields \eqref{eq:local-solve-accuracy}. The returned point is interior and is therefore a valid warm start for the next call.

Finally, since we set \(\varepsilon_{a,1}=0\) and \(\varepsilon_{a,t}=R/T\) for \(2\leq t\leq T\), we have \(E_{a,T}=(T-1)\frac RT<R.\) Averaging over agents proves \(E_T < R\).
\end{proof}

\section{Proof of Lemma~\ref{lem:feasible-point-disagreement}}
\begin{proof}[Proof of Lemma~\ref{lem:feasible-point-disagreement}]
We derive two properties of $p_{\mathcal C}(\cdot)$.
By Lemma~\ref{lem:gauge-properties}, the inclusion \(r\mathbb B_2\subseteq\mathcal C\) in Assumption~\ref{ass:geometry} makes \(\gamma_{\mathcal C}\), and hence \(S_{\mathcal C}\), \(1/r\)-Lipschitz continuous. 
For $u,v\in R\mathbb B_2$, set
\[
    A:=1+S_{\mathcal C}(u),
    \qquad
    B:=1+S_{\mathcal C}(v),
\]
By definition, $p_{\mathcal C}(u)=\frac{u}{A}$, $p_{\mathcal C}(v)=\frac{v}{B}.$ Adding and subtracting $v/A$ and applying the triangle inequality yields
\begin{align*}
\|p_{\mathcal C}(u)-p_{\mathcal C}(v)\|_2
&= \left\|\frac{u}{A}-\frac{v}{B}\right\|_2 
= \left\| \frac{u-v}{A} + v\left(\frac{1}{A}-\frac{1}{B}\right) \right\|_2 \\
&\le \frac{\|u-v\|_2}{A} + \|v\|_2 \left|\frac{1}{A}-\frac{1}{B}\right| \notag\\
&= \frac{\|u-v\|_2}{A} + \frac{\|v\|_2|A-B|}{AB}.
\end{align*}
Since $S_{\mathcal C}\ge 0$, we have $A,B\ge 1$. Moreover,
$v\in R\mathbb B_2$ implies $\|v\|_2\le R$, and the
$1/r$-Lipschitz continuity of $S_{\mathcal C}$ gives
\(
    |A-B|
    = |S_{\mathcal C}(u)-S_{\mathcal C}(v)|
    \le \frac{1}{r}\|u-v\|_2.
\)
Substituting these bounds yields
\begin{align}
\|p_{\mathcal C}(u)-p_{\mathcal C}(v)\|_2
&\le
\|u-v\|_2
+
\frac{R}{r}\|u-v\|_2 \notag\\
&=
(1+\kappa)\|u-v\|_2,
\label{eq:radial-projection-lipschitz}
\end{align}
where $\kappa=R/r$.
Appying Lemma~\ref{lem:gauge-proj-prop} with
$\varepsilon_{\rm gau}=1/T$ gives \(S_{\mathcal C}(u_{a,t}) \le S_{a,t} \le S_{\mathcal C}(u_{a,t})+\frac{1}{T},\)
and therefore
\[
    |S_{a,t}-S_{\mathcal C}(u_{a,t})|
    \le \frac{1}{T}.
\]
Moreover, by definition, \(p_{a,t} = \frac{u_{a,t}}{1+S_{a,t}},\) and \(p_{\mathcal C}(u_{a,t}) = \frac{u_{a,t}}{1+S_{\mathcal C}(u_{a,t})}.\)
Hence
\begin{align}
\|p_{a,t}-p_{\mathcal C}(u_{a,t})\|_2
&=
\left\| u_{a,t}
\left(
    \frac{1}{1+S_{a,t}} - \frac{1}{1+S_{\mathcal C}(u_{a,t})}
\right)
\right\|_2
=
\frac{
    \|u_{a,t}\|_2
    |S_{a,t}-S_{\mathcal C}(u_{a,t})|
}{
    (1+S_{a,t})
    (1+S_{\mathcal C}(u_{a,t}))
}
\notag\\
&\le
\|u_{a,t}\|_2
|S_{a,t}-S_{\mathcal C}(u_{a,t})|
\le \frac{R}{T},
\label{eq:radial-projection-error}
\end{align}
where the first inequality uses $S_{a,t}\ge 0$ and $S_{\mathcal C}(u_{a,t})\ge 0$ so that the denominator is at least one, and the last inequality uses $\|u_{a,t}\|_2\le R$ and $|S_{a,t}-S_{\mathcal C}(u_{a,t})|\le 1/T$.

By definitions, $w_{a,t}=x_0+p_{a,t}$ and $z_t=x_0+p_{\mathcal C}(\bar u_t^\star).$
Hence the common $x_0$ cancels, and
\begin{align*}
\|w_{a,t}-z_t\|_2
& = \|p_{a,t}-p_{\mathcal C}(\bar u_t^\star)\|_2 \\
&=
\left\|
    p_{a,t}
    -p_{\mathcal C}(u_{a,t})
    +p_{\mathcal C}(u_{a,t})
    -p_{\mathcal C}(\bar u_t^\star)
\right\|_2\\
&\le
\|p_{a,t}-p_{\mathcal C}(u_{a,t})\|_2
+
\|p_{\mathcal C}(u_{a,t})
-p_{\mathcal C}(\bar u_t^\star)\|_2 \\
& \le \frac{R}{T} +(1+\kappa) \|u_{a,t}-\bar u_t^\star\|_2.
\end{align*}
where the final inequality comes from \eqref{eq:radial-projection-error} and \eqref{eq:radial-projection-lipschitz}.
By the triangle inequality, Lemmas~\ref{lem:local-solve-accuracy} and~\ref{lem:minimizer-disagreement} yield
\begin{align*}
\|u_{a,t}-\bar u_t^\star\|_2
\le \|u_{a,t}-u_{a,t}^\star\|_2 + \|u_{a,t}^\star-\bar u_t^\star\|_2
\le \varepsilon_{a,t} + \frac{d_{a,t}}{m_T},
\end{align*}
Therefore,
\[
    \|w_{a,t}-z_t\|_2
    \le
    \frac{R}{T}
    +(1+\kappa)
    \left(
        \varepsilon_{a,t}
        +\frac{d_{a,t}}{m_T}
    \right).
\]

Applying the triangle inequality through \(z_t\), averaging over payoff owners \(j\), and summing over \(t\) gives
\[
\begin{aligned}
    \sum_{t=1}^T\frac1N\sum_{j=1}^N
    \|w_{a,t}-w_{j,t}\|_2
    \leq{}&
    2R+(1+\kappa)(E_{a,T}+E_T)
    +
    \frac{1+\kappa}{m_T}
    \left(
    \sum_{t=1}^T d_{a,t}
    +
    \sum_{t=1}^T\frac1N\sum_{j=1}^N d_{j,t}
    \right).
\end{aligned}
\]
Using Equations~\eqref{eq:row-state-sum} and
\eqref{eq:average-state-sum} to bound $\sum_{t=1}^T d_{a,t}$ and $\sum_{t=1}^T\frac1N\sum_{j=1}^N d_{j,t}$ respectively to obtain the final bound as stated.
\end{proof}

%================================================================

\section{Applications to Up-Concave and DR-Submodular Maximization}
\label{sec:applications}

We instantiate Theorem~\ref{thm:main} through wrapper pairs covering four application families. The upper-linearization inequalities are inherited from prior work, and our consequence is their combination with the per-agent decentralized separation-oracle guarantee. Throughout this section, the action set also satisfies Assumption~\ref{ass:geometry}, which states for a known $x_0\in\operatorname{int}(\mathcal K)$,
\[
 \mathcal C=\mathcal K-x_0,
 \qquad
 r\mathbb B_2\subseteq\mathcal C\subseteq R\mathbb B_2,
 \qquad
 \kappa=R/r.
\]
Thus an application-specific requirement that $\mathbf0\in\mathcal K$ does not replace the interior center $x_0$ used by the gauge algorithm.

\paragraph{Function classes.}
For vectors, $x\leq y$ denotes coordinatewise order, and $\|x\|_\infty:=\max_k|x_k|$. A differentiable function $f$ is $\gamma$-weakly up-concave, for $\gamma\in(0,1]$, if for every $x\leq y$ in its domain,
\[
 \gamma\langle\nabla f(y),y-x\rangle
 \leq f(y)-f(x)
 \leq\frac1\gamma\langle\nabla f(x),y-x\rangle.
\]
We use \emph{up-concave} when $\gamma=1$. For a continuous monotone function, its curvature is the smallest $c\in[0,1]$ such that
\[
 f(y+z)-f(y)
 \geq(1-c)\bigl(f(x+z)-f(x)\bigr)
\]
whenever $x,y,z\geq\mathbf0$ and all four displayed arguments belong to the domain. These conventions follow \citet{pedramfar2024linear}. A differentiable function is continuous DR-submodular if $\nabla f(x)\geq\nabla f(y)$ coordinatewise whenever $x\leq y$. A set $\mathcal K\subseteq\mathbb R_{\geq0}^d$ is down-closed if $x\in\mathcal K$ and $\mathbf0\leq y\leq x$ imply $y\in\mathcal K$.

\begin{assumption}[One-query first-order base oracle]
\label{ass:one-query-interface}
There exists $L>0$ such that, for every differentiable payoff $f$ and query point $z\in\mathcal K$, the base oracle $\mathcal O_f(z)$ returns a random vector $V_f(z)$ satisfying, conditional on the previously revealed information and on the chosen query point,
\[
 \mathbb E[V_f(z)\mid \text{past and }z]=\nabla f(z),
 \qquad
 \|V_f(z)\|_2\leq L.
\]
The auxiliary randomness used by the Query Wrapper is sampled as part of the wrapped-response generation, after conditioning on $\mathcal F_t^{-}$, and independently of the subsequent base-oracle noise.
\end{assumption}

Each wrapper below invokes its base oracle exactly once. Following the feedback convention of \citet{pedramfar2025uniform} and \citet{lu2025decentralized}, the feedback is \emph{first-order semi-bandit} when the base-oracle query point equals the played action $\mathcal W^{\mathrm{act}}(w)$, and \emph{first-order full-information} when the query point may differ from that action. Here full-information means that one first-order query may be placed at another feasible point; it does not mean that the entire payoff function is revealed. We do not use blocking, smoothing, zeroth-order estimators, or other limited-feedback conversions.

For a wrapper $\mathcal W=(\mathcal W^{\rm act},\mathcal W^{\rm qry})$, we write $\mathcal W^{\rm qry}(\mathcal O_f)$ as the vector-valued oracle obtained by applying the wrapper to the base oracle $\mathcal O_f$. When invoked at $w$, it returns
\[
    q\sim \mathcal W^{\rm qry}(\mathcal O_f)(w).
\]
When the wrapper is applied to $f_{t,j}$ at $w_{j,t}$, its proxy is
\[
    g(f_{t,j},w_{j,t})
    :=
    \mathbb E[q_{j,t}\mid\mathcal F_t^-].
\]
The tower property over the wrapper and base-oracle randomness verifies the conditional-mean requirement in Assumption~\ref{ass:feedback}. All four wrapped outputs satisfy $\|q\|_2\leq L$, so $G=L$. Moreover,
\[
 \|\nabla f(z)\|_2
 =\|\mathbb E[V_f(z)\mid \text{past and }z]\|_2
 \leq L.
\]
Hence $f$ is $L$-Lipschitz on the convex action set. If the Action Wrapper $h=\mathcal W^{\mathrm{act}}$ is $L_h$-Lipschitz, then
\begin{equation}
 |f(h(w))-f(h(v))|
 \leq LL_h\|w-v\|_2.
 \label{eq:application-payoff-stability}
\end{equation}
so Assumption~\ref{ass:payoff-stability} holds with $L_\phi=LL_h$. Substituting $G=L$ and $L_\phi=LL_h$ into Theorem~\ref{thm:main} gives, for every agent $a$,
\begin{equation}
 \mathfrak R_\alpha^a(T)
 =\widetilde O\!\left(
 \beta+LR\left[
 \beta\kappa+L_h(1+\kappa)(1+\log N)
 \right]
 \sqrt{\frac{T}{1-\rho}}
 \right).
 \label{eq:application-regret-template}
\end{equation}
In each corollary below, every local payoff $f_{t,j}$ is assumed to satisfy the corresponding proposition with common class and oracle constants, and the same wrapper is used by every agent and round. In the anchored non-monotone case, the anchor $\underline x$ is also fixed and shared.
Every wrapper makes one local payoff-oracle query and no action-set-oracle call. Thus the application instantiations retain one neighbor exchange per round and the per-agent $O(T\log(\kappa T))$ separation-oracle count of Theorem~\ref{thm:main}.

\subsection{Monotone Up-Concave Functions over General Convex Sets}

\begin{proposition}[Identity wrapper for monotone weakly up-concave functions; adapted from \citet{pedramfar2024linear}]
\label{prop:mono-general-wrapper}
Let $f:[0,1]^d\to\mathbb R_{\geq0}$ be differentiable, monotone, and $\gamma$-weakly up-concave for $\gamma\in(0,1]$, with curvature at most $c\in[0,1]$, and let $\mathcal K\subseteq[0,1]^d$ be convex. Define
\[
 \mathcal W^{\mathrm{act}}(w):=h(w)=w,
 \qquad
 \mathcal W^{\mathrm{qry}}(\mathcal O_f)(w)
 :=\mathcal O_f(w).
\]
Thus $q=V_f(w)$ and $\mathfrak g(f,w)=\nabla f(w)$. The base-oracle query equals the played action, so this is first-order semi-bandit feedback. For all $w,y\in\mathcal K$, 
\[
 \frac{\gamma^2}{1+c\gamma^2}f(y)-f(h(w))
 \leq
 \frac{\gamma}{1+c\gamma^2}
 \langle\mathfrak g(f,w),y-w\rangle.
\]
Consequently, this wrapper instantiates
\[
 (\alpha,\beta,G,L_\phi)
 =
 \left(
 \frac{\gamma^2}{1+c\gamma^2},
 \frac{\gamma}{1+c\gamma^2},
 L,L
 \right).
\]
\end{proposition}
\begin{proof}
The action and query point both equal $w\in\mathcal K$. Assumption~\ref{ass:one-query-interface} gives $\mathfrak g(f,w)=\nabla f(w)$ and $\|q\|_2\leq L$. The upper-linearization inequality, including its curvature-dependent constants, is \citet[Lemma~1, specialized to $\mu=0$]{pedramfar2024linear}. Finally, the identity Action Wrapper is $1$-Lipschitz, so \eqref{eq:application-payoff-stability} gives $L_\phi=L$.
\end{proof}

\begin{corollary}[Monotone up-concave regret on a general set]
\label{cor:mono_general}
Under Proposition~\ref{prop:mono-general-wrapper} and the remaining assumptions of Theorem~\ref{thm:main}, every agent $a\in[N]$ satisfies
\[
 \mathfrak R_{\gamma^2/(1+c\gamma^2)}^a(T)
 =\widetilde O\!\left(
 \frac{\gamma}{1+c\gamma^2}
 +LR\left[
 \frac{\gamma\kappa}{1+c\gamma^2}
 +(1+\kappa)(1+\log N)
 \right]
 \sqrt{\frac{T}{1-\rho}}
 \right).
\]
\end{corollary}
\begin{proof}
Apply \eqref{eq:application-regret-template} with the tuple in Proposition~\ref{prop:mono-general-wrapper} and $L_h=1$.
\end{proof}

\subsection{Monotone Up-Concave Functions over Convex Sets Containing the Origin}

When $\mathbf0\in\mathcal K$, a randomized Query Wrapper yields the curvature-independent coefficient $1-e^{-\gamma}$.

\begin{proposition}[Boosted wrapper for monotone weakly up-concave functions; adapted from \citet{pedramfar2024linear}]
\label{prop:mono-origin-wrapper}
Let $f:[0,1]^d\to\mathbb R_{\geq0}$ be differentiable, monotone, and $\gamma$-weakly up-concave for $\gamma\in(0,1]$, and let $\mathcal K\subseteq[0,1]^d$ be convex with $\mathbf0\in\mathcal K$. Define $\mathcal W^{\mathrm{act}}(w)=h(w)=w$. The Query Wrapper $\mathcal W^{\mathrm{qry}}(\mathcal O_f)(w)$ draws $Z\in[0,1]$ with density
\[
 p_\gamma(z)=\frac{\gamma e^{\gamma(z-1)}}{1-e^{-\gamma}}
\]
and returns
\[
 q=V_f(Zw).
\]
Then
\[
 \mathfrak g(f,w)
 =\int_0^1p_\gamma(z)\nabla f(zw)\,dz.
\]
The query $Zw$ generally differs from the played action $w$, so this is first-order full-information feedback. For all $w,y\in\mathcal K$, 
\[
 (1-e^{-\gamma})f(y)-f(h(w))
 \leq
 \frac{1-e^{-\gamma}}{\gamma}
 \langle\mathfrak g(f,w),y-w\rangle.
\]
Consequently,
\[
 (\alpha,\beta,G,L_\phi)
 =
 \left(
 1-e^{-\gamma},
 \frac{1-e^{-\gamma}}{\gamma},
 L,L
 \right).
\]
\end{proposition}
\begin{proof}
Convexity and $\mathbf0\in\mathcal K$ imply $Zw\in\mathcal K$. Conditional unbiasedness of the base oracle followed by the tower property over $Z$ gives the displayed proxy, while $\|q\|_2\leq L$. The boosted upper-linearization inequality is \citet[Lemma~2 and Algorithm~2]{pedramfar2024linear}. The identity Action Wrapper has $L_h=1$, so $L_\phi=L$.
\end{proof}

\begin{corollary}[Monotone up-concave regret when $\mathbf0\in\mathcal K$]
\label{cor:mono_0_in_K}
Under Proposition~\ref{prop:mono-origin-wrapper} and the remaining assumptions of Theorem~\ref{thm:main}, every agent $a\in[N]$ satisfies
\[
 \mathfrak R_{1-e^{-\gamma}}^a(T)
 =\widetilde O\!\left(
 \frac{1-e^{-\gamma}}{\gamma}
 +LR\left[
 \frac{(1-e^{-\gamma})\kappa}{\gamma}
 +(1+\kappa)(1+\log N)
 \right]
 \sqrt{\frac{T}{1-\rho}}
 \right).
\]
\end{corollary}
\begin{proof}
Apply \eqref{eq:application-regret-template} with the tuple in Proposition~\ref{prop:mono-origin-wrapper} and $L_h=1$.
\end{proof}

\subsection{Non-Monotone Up-Concave Functions over General Convex Sets}

For a non-monotone payoff, the Action Wrapper contracts the internal point toward a fixed feasible anchor.

\begin{proposition}[Anchored wrapper for non-monotone up-concave functions; adapted from \citet{pedramfar2024linear}]
\label{prop:nonmono-general-wrapper}
Let $f:[0,1]^d\to\mathbb R_{\geq0}$ be differentiable and up-concave, without a monotonicity assumption, and let $\mathcal K\subseteq[0,1]^d$ be convex. Fix a known $\underline x\in\mathcal K$ with $\|\underline x\|_\infty<1$ and define
\[
 \mathcal W^{\mathrm{act}}(w)=h(w)=\frac{w+\underline x}{2}.
\]
The Query Wrapper $\mathcal W^{\mathrm{qry}}(\mathcal O_f)(w)$ draws $Z\in[0,1]$ with density
\[
 p_{\rm nm}(z)=\frac{1}{3(1-z/2)^3}
\]
and returns
\[
 q=V_f\!\left(\underline x+\frac Z2(w-\underline x)\right).
\]
Then
\[
 \mathfrak g(f,w)
 =\int_0^1p_{\rm nm}(z)
 \nabla f\!\left(\underline x+\frac z2(w-\underline x)\right)\,dz.
\]
The query point generally differs from the played action $(w+\underline x)/2$, so this is first-order full-information feedback. For all $w,y\in\mathcal K$,
\[
 \frac{1-\|\underline x\|_\infty}{4}f(y)-f(h(w))
 \leq
 \frac38\langle\mathfrak g(f,w),y-w\rangle.
\]
Consequently,
\[
 (\alpha,\beta,G,L_\phi)
 =
 \left(
 \frac{1-\|\underline x\|_\infty}{4},
 \frac38,
 L,\frac L2
 \right).
\]
Under Assumption~\ref{ass:geometry}, the known interior point $x_0$ is an admissible choice of $\underline x$.
\end{proposition}
\begin{proof}
The played action and the query point are convex combinations of $w$ and $\underline x$, hence both belong to $\mathcal K$. The density is normalized because
\[
 \int_0^1\frac{dz}{3(1-z/2)^3}=1.
\]
Conditional unbiasedness followed by the tower property over $Z$ gives the displayed proxy without increasing the almost-sure norm bound. The structural inequality is \citet[Lemma~3 and Algorithm~3]{pedramfar2024linear}. The Action Wrapper is $1/2$-Lipschitz, so $L_\phi=L/2$. Finally, because $\mathcal K\subseteq[0,1]^d$ is full-dimensional and $x_0\in\operatorname{int}(\mathcal K)$, every coordinate of $x_0$ is strictly below $1$, and hence $\|x_0\|_\infty<1$.
\end{proof}

\begin{corollary}[Non-monotone up-concave regret on a general set]
\label{cor:nonmono_general}
Under Proposition~\ref{prop:nonmono-general-wrapper} and the remaining assumptions of Theorem~\ref{thm:main}, every agent $a\in[N]$ satisfies
\[
 \mathfrak R_{(1-\|\underline x\|_\infty)/4}^a(T)
 =\widetilde O\!\left(
 \frac38
 +LR\left[
 \frac{3\kappa}{8}
 +\frac{(1+\kappa)(1+\log N)}{2}
 \right]
 \sqrt{\frac{T}{1-\rho}}
 \right).
\]
\end{corollary}
\begin{proof}
Apply \eqref{eq:application-regret-template} with the tuple in Proposition~\ref{prop:nonmono-general-wrapper} and $L_h=1/2$.
\end{proof}

\subsection{Non-Monotone DR-Submodular Functions on Down-Closed Sets}

\begin{proposition}[Exponential wrapper for non-monotone DR-submodular functions; adapted from \citet{lu2026upper}]
\label{prop:dr-wrapper}
Let $f:[0,1]^d\to\mathbb R_{\geq0}$ be differentiable, non-monotone, and DR-submodular, and let $\mathcal K\subseteq[0,1]^d$ be convex, down-closed, and contain the origin. Define
\[
 \mathcal W^{\mathrm{act}}(w)=h(w)=\mathbf1-e^{-w}.
\]
The Query Wrapper $\mathcal W^{\mathrm{qry}}(\mathcal O_f)(w)$ draws $Z\in[0,1]$ with density
\[
 p(z)=\frac{e^{z-1}}{1-e^{-1}},
\]
sets $z^\circ=\mathbf1-e^{-Zw}$, and returns
\[
 q=V_f(z^\circ)\odot e^{-Zw}.
\]
Define
\[
 F_f(w)=\int_0^1
 \frac{e^{z-1}}{(1-e^{-1})z}
 \bigl(f(\mathbf1-e^{-zw})-f(\mathbf0)\bigr)\,dz.
\]
The integrand at $z=0$ is interpreted by its continuous extension.
Then
\[
 \mathfrak g(f,w)
 =\mathbb E_Z\!\left[
 \nabla f(\mathbf1-e^{-Zw})\odot e^{-Zw}
 \right]
 =\nabla F_f(w).
\]
The query $z^\circ$ generally differs from the played action $\mathbf1-e^{-w}$, so this is first-order full-information feedback. For all $w,y\in\mathcal K$,
\[
 \frac1e f(y)-f(h(w))
 \leq
 (1-e^{-1})\langle\mathfrak g(f,w),y-w\rangle.
\]
Consequently,
\[
 (\alpha,\beta,G,L_\phi)
 =\left(\frac1e,1-e^{-1},L,L\right).
\]
\end{proposition}
\begin{proof}
For $w\in[0,1]^d$ and $Z\in[0,1]$,
\[
 \mathbf1-e^{-Zw}\leq Zw\leq w,
 \qquad
 \mathbf1-e^{-w}\leq w
\]
coordinatewise. Down-closedness therefore gives $z^\circ,h(w)\in\mathcal K$. Since the coordinates of $e^{-Zw}$ lie in $(0,1]$,
\[
 \|q\|_2
 \leq\|V_f(z^\circ)\|_2
 \leq L.
\]
Conditional unbiasedness and the tower property over $Z$ give the displayed proxy. The bounded-gradient condition supplies an integrable dominating function, so differentiation under the integral gives $\mathfrak g(f,w)=\nabla F_f(w)$. The $1/e$ upper-linearization inequality is \citet[Theorem~1]{lu2026upper}, and its expectation representation and one-query estimator appear in that paper's Remark~3, Algorithm~1, and Lemma~2. Finally, the Jacobian of $h(w)=\mathbf1-e^{-w}$ is diagonal with operator norm at most $1$ on $[0,1]^d$, so $L_\phi=L$ is valid.
\end{proof}

\begin{corollary}[Non-monotone DR-Submodular over Down-closed Set]
\label{cor:nonmono_downclosed}
Under Proposition~\ref{prop:dr-wrapper} and the remaining assumptions of Theorem~\ref{thm:main}, every agent $a\in[N]$ satisfies
\[
 \mathfrak R_{1/e}^a(T)
 =\widetilde O\!\left(
 1-e^{-1}
 +LR\left[
 (1-e^{-1})\kappa
 +(1+\kappa)(1+\log N)
 \right]
 \sqrt{\frac{T}{1-\rho}}
 \right).
\]
\end{corollary}
\begin{proof}
Apply \eqref{eq:application-regret-template} with the tuple in Proposition~\ref{prop:dr-wrapper} and $L_h=1$.
\end{proof}

The four wrappers above are not new approximation reductions. Their role here is to provide the exact Action and Query Wrappers, verify the feedback and payoff-stability assumptions of Theorem~\ref{thm:main}, and obtain the corresponding decentralized guarantees.

\end{document}